\documentclass[psamsfonts]{amsart}
\pdfoutput=1

\usepackage{amssymb,amsfonts}
\usepackage{amsmath}
\usepackage[all,arc]{xy}
\usepackage{enumerate}
\usepackage{mathrsfs}
\usepackage{graphicx}
\usepackage{tikz}
\usepackage{tikz-cd}
\usepackage{eufrak}
\usepackage[utf8]{inputenc}
\usepackage{mathtools}
\usepackage{stmaryrd}
\usepackage{enumitem}
\usepackage[T1]{fontenc} 
\usepackage{marginnote} 
\usepackage{caption}
\usepackage{hyperref}
\hypersetup{hidelinks}
\usepackage{changepage}

\newtheorem{thm}{Theorem}[section]
\newtheorem{cor}[thm]{Corollary}
\newtheorem{prop}[thm]{Proposition}

\theoremstyle{definition}
\newtheorem{defn}[thm]{Definition}

\newtheorem{exmp}[thm]{Example}

\theoremstyle{remark}

\makeatletter
\let\c@equation\c@thm
\makeatother
\numberwithin{equation}{section}

\title[Fox-Neuwirth cells, quantum shuffle algebras, and resultant]{Fox-Neuwirth cells, quantum shuffle algebras, and character sums of the resultant}

\author{Anh Trong Nam Hoang}
\address{School of Mathematics, University of Minnesota, Minneapolis, MN}
\email{hoang278@umn.edu}

\date{\today}

\begin{document}

\begin{abstract}
We give an upper bound on character sums of the resultant over pairs of monic square-free polynomials of given degrees, answering a question of Ellenberg and Shusterman in the quadratic case. Our approach is topological: we compute the homology of braid groups on multi-punctured planes and prove a vanishing range for the homology of mixed braid groups with rank-$1$ local coefficients associated to characters of finite fields. Our method involves constructing a cellular stratification for configuration spaces of multi-punctured planes and relating their twisted homology with more general exponential coefficients to the cohomology of certain bimodules over quantum shuffle algebras.
\end{abstract}

\maketitle

\tableofcontents


\section{Introduction}

Given two monic polynomials $f$ and $g$ in one variable of positive degrees over a field $\mathbf{k}$, their {\it resultant} $\mathcal{R}(f,g)$ is given by the product formula
\[ \mathcal{R}(f,g) := \prod_{i,j}(x_i - y_j) \]
where $x_1,\dots,x_n$ are roots of $f$ and $y_1,\dots,y_m$ are roots of $g$ in any algebraically closed extension of $\mathbf{k}$. Alternatively, if
\[f(x) = x^n + a_{n-1}x^{n-1} + \dots + a_1x+a_0\]
and
\[g(x) = x^m + b_{m-1}x^{m-1} + \dots + b_1x+b_0,\]
then $\mathcal{R}(f,g)$ is an integer-coefficient polynomial expression of the coefficients $a_i$ and $b_j$, which can be computed by taking the determinant of the polynomials' Sylvester matrix:
\[\mathcal{R}(f,g)=
\begin{vmatrix}
a_0 & a_1 & \dots & a_{n-1} & 1 & 0 & \dots & 0\\
0 & a_0 & a_1 & \dots & a_{n-1} & 1 & \dots & 0\\
\vdots & \ddots & \ddots & \ddots & \ddots & \ddots & \ddots & \vdots\\
0 & 0 & \dots & a_0 & a_1 & \dots & a_{n-1} & 1\\
b_0 & b_1 & \dots & b_{m-1} & 1 & 0 & \dots & 0\\
0 & b_0 & b_1 & \dots & b_{m-1} & 1 & \dots & 0\\
\vdots & \ddots & \ddots & \ddots & \ddots & \ddots & \ddots & \vdots\\
0 & 0 & \dots & b_0 & b_1 & \dots & b_{m-1} & 1\\
\end{vmatrix}.\]

The defining property of the resultant lies in the following fact: two polynomials $f$ and $g$ share a common root if and only if $\mathcal{R}(f,g)=0$. This classical concept was first introduced by Cayley \cite{cayley} in the context of elimination theory and has since found abundant applications in generalizing the notion of discriminants, solving systems of polynomial equations, and proving important theorems such as Bezout's Theorem and Hilbert's Nullstellensatz. A beautiful survey on the general theory of resultants from this perspective can be found in \cite{gelfand-kapranov-zelevinsky}. Over a field $\mathbf{k}$, the resultant can be interpreted as a map $\mathcal{R}: \mathbb{A}^n_\mathbf{k} \times \mathbb{A}^m_\mathbf{k} \to \mathbb{A}^1_\mathbf{k}$, where the affine spaces $\mathbb{A}^n_\mathbf{k}$ and $\mathbb{A}^m_\mathbf{k}$ may be identified with spaces of monic degree-$n$ and degree-$m$ polynomials over $\mathbf{k}$. The resultant locus $\mathbb{A}^n_\mathbf{k}\times\mathbb{A}^m_\mathbf{k}\setminus\mathcal{R}^{-1}(0)$, namely the space of pairs of monic coprime polynomials of degrees $n$ and $m$, is a classically studied object, whereas the topology and arithmetic of the hypersurface $\mathcal{R}^{-1}(1)$ were recently studied by \cite{fw17} when $n=m$.

Fix a finite field $\mathbb{F}_q$, a prime $\ell$ invertible in $\mathbb{F}_q$, and a choice of inclusion $\overline{\mathbb{Q}}_\ell \subseteq \mathbb{C}$. Let $\chi : \mathbb{F}_q \to \mathbb{C}$ be a nontrivial character. In this paper, we are concerned with the {\it character sum}
\[ F_\chi(n,m,q) := \sum_{f,g} \chi(\mathcal{R}(f,g))\]
where $f$ and $g$ range over monic {\it square-free} polynomials of degrees $n$ and $m$ over $\mathbb{F}_q$. This sum was first considered by Ellenberg and Shusterman in the case when $\chi$ is the quadratic character \cite{ellenberg-birs}. Characters of the resultant over finite fields often appear in the form of residue symbols (e.g., the Jacobi symbol when the associated character is quadratic) and therefore play a central role in recent works on quadratic reciprocity \cite{clark-pollack} as well as Dirichlet series and M\"{o}bius functions over function fields \cite{sawin-shusterman1,sawin-shusterman2,sawin22}.

The goal of this paper is to produce an upper bound on the character sum $F_\chi(n,m,q)$ as a function of $q$:

\begin{thm}\label{thm:main_intro}
    Let $\chi$ be a nontrivial character of $\mathbb{F}_q$. Then
    \[|F_\chi(n,m,q)| \le 2^{2n+2m-1}\frac{q^{n+m+1-\mathrm{max}(n,m)/2}-1}{\sqrt{q}-1}\]
    if $\chi$ is quadratic, and
    \[|F_\chi(n,m,q)| \le 2^{2n+2m-1}\frac{q^{n+m+(1-\mathrm{max}(n,m))/2}-1}{\sqrt{q}-1}\]
    otherwise.
\end{thm}

This result is proved as Theorem~\ref{thm:main_nb}.
As a consequence, we make the following observation about the asymptotic behavior of character averages of the resultant (Corollary~\ref{cor:res_avg}).

\begin{cor}
   For a sufficiently large $q$, the asymptotic average of a nontrivial character of the resultant over pairs of monic square-free polynomials over $\mathbb{F}_q$ approaches $0$ as the degree of either or both polynomials grows indefinitely. 
\end{cor}

Our method follows a young program of applying topological studies of {\it configuration spaces} to problems in arithmetic statistics, via the Grothendieck-Lefschetz trace formula with twisted coefficients and various comparison theorems in \'{e}tale cohomology theory (see, e.g., \cite{cef_repstab,fww19}). Notable examples include work by Ellenberg--Venkatesh--Westerland on the Cohen-Lenstra conjecture over function fields \cite{evw1} and work by Bergstr\"{o}m--Diaconu--Petersen--Westerland on the asymptotics of moments of quadratic $L$-functions \cite{bdpw23}. In recent work that inspired the approach of this paper, Ellenberg--Tran--Westerland proved that the upper bound in the weak Malle’s conjecture over $\mathbb{F}_q(t)$ holds for all choices of the Galois group and all sufficiently large $q$, by producing a bound on the homology of configuration spaces with certain exponential coefficients \cite{etw17}.


\subsection{Outline of the argument}

Let $\mathrm{Conf}_{n,m}$ denote the space of pairs of monic square-free coprime polynomials of degrees $n$ and $m$. Observe that the sum $F_\chi(n,m,q)$ is the same as
\[ F_\chi(n,m,q) = \sum_{(f,g)\in \mathrm{Conf}_{n,m}(\mathbb{F}_q)} \chi(\mathcal{R}(f,g))\]
since $\mathcal{R}(f,g)=0$ whenever $f$ and $g$ share a common root. By pulling back the Kummer sheaf $\mathcal{L}_\chi$ associated to $\chi$ along the resultant map $\mathcal{R}:\mathrm{Conf}_{n,m}\to\mathbb{A}^1\setminus\{0\}$, we obtain a rank-$1$ local system $\mathcal{R}^*\mathcal{L}_\chi$ on $\mathrm{Conf}_{n,m}$ with the property that $\mathrm{tr}(\mathrm{Frob}_q|(\mathcal{R}^*\mathcal{L}_{\chi})_{(f,g)}) = \chi(\mathcal{R}(f,g))$ for all $(f,g) \in \mathrm{Conf}_{n,m}$. By the Grothendieck-Lefschetz trace formula with twisted coefficients and Artin's comparison theorem, $F_\chi(n,m,q)$ can then be approached by studying the cohomology of the {\it bicolor configuration space} $\mathrm{Conf}_{n,m}(\mathbb{C})$, the space of configurations of $n$ blue and $m$ red points on the plane, with coefficients in $\mathcal{R}^*\mathcal{L}_\chi$. This viewpoint was first introduced by Ellenberg and Shusterman when $\chi$ is the quadratic character \cite{ellenberg-birs}. The bulk of this paper is dedicated to proving a vanishing range for these cohomology groups, which will produce the upper bounds on $F_\chi(n,m,q)$ in Theorem~\ref{thm:main_intro}.

Recall that the $n^{\mathrm{th}}$ \textit{(unordered) configuration space} of a topological space $X$ is defined to be
\[\mathrm{Conf}_n(X) := \{ \{ x_1,\dots,x_n \} \subset X : x_i \ne x_j \text{ if } i \ne j \}.\]
To study the homology of the bicolor configuration space $\mathrm{Conf}_{n,m}(\mathbb{C})$, by passing through a fiber sequence of Fadell--Neuwirth \cite{fadneu62} and the Lyndon-Hochschild-Serre spectral sequence, it suffices to study that of the $n^\mathrm{th}$ configuration space of the plane with $m$ punctures, $\mathrm{Conf}_n(\mathbb{C}_m)$. Sections~\ref{sec:fnf} and~\ref{subsec:fnf_twisted} will focus on constructing a cellular stratification for $\mathrm{Conf}_n(\mathbb{C}_m)$ and developing a framework for computing their cellular homology with coefficients in arbitrary local systems. Our construction extends a similar framework for $\mathrm{Conf}_n(\mathbb{C}_1)$ given by \cite{h22}, which is based on the classical Fox-Neuwirth cellular stratification of $\mathrm{Conf}_n(\mathbb{C})$ established by \cite{fn62,fuk70} and a twisted cellular chain complex given by \cite{etw17}.

In Section~\ref{sec:rep_br}, we will recall several concepts in quantum algebra and develop a representation theory of braid subgroups that are necessary to state our main topological results. Recall that a \textit{braided vector space} over a field $\mathbf{k}$ is a finite dimensional vector space $V$ equipped with an automorphism $\sigma: V \otimes V \to V \otimes V$ that satisfies the braid equation
\[(\sigma \otimes \mathrm{id}) \circ (\mathrm{id} \otimes \sigma) \circ (\sigma \otimes \mathrm{id}) = (\mathrm{id} \otimes \sigma) \circ (\sigma \otimes \mathrm{id}) \circ (\mathrm{id} \otimes \sigma)\]
on $V^{\otimes 3}$. There is a natural action of the braid group $B_n$ on $V^{\otimes n}$. The \textit{quantum shuffle algebra} $\mathfrak{A}(V)$ is a braided, graded Hopf algebra whose underlying coalgebra is the cofree coalgebra on $V$ and multiplication is given by a shuffle product involving the braiding $\sigma$ (Definition~\ref{defn:qsa}). Notably, shuffle algebras have traditionally been connected to the homology of configuration spaces with coefficients in local systems; see, e.g., \cite{fuk70,vai78,mar96,cal06,ks20,etw17}.

Given a pair of braided vector spaces $(V,W)$, under some fairly relaxed conditions (Definitions~\ref{defn:mix_bvs} and~\ref{defn:sep_mbvs}), there is an action of the {\it mixed braid group} $B_{n,m} := \pi_1(\mathrm{Conf}_{n,m}(\mathbb{C}))$ on $V^{\otimes n}\otimes W^{\otimes m}$, which restricts to an action of the surface braid group $B_n(\mathbb{C}_m):=\pi_1(\mathrm{Conf}_n(\mathbb{C}_m))$. In \cite{h22}, we defined a bimodule $\mathfrak{M}$ over the quantum shuffle algebra $\mathfrak{A}(V)$ by
\[ \mathfrak{M} = \displaystyle \bigoplus_{q \ge 1} \bigoplus_{0 \le j \le q-1} V^{\otimes j} \otimes W \otimes V^{\otimes q-j-1} \]
with multiplication resembling the quantum shuffle product (Definition~\ref{defn:M}) and related its cohomology to the homology of $\mathrm{Conf}_n(\mathbb{C}_1)$ with coefficients in the local system associated to the $B_n(\mathbb{C}_1)$-representation on $V^{\otimes n}\otimes W$. Our main topological theorem extends this result, expressing the (co)homology of $\mathrm{Conf}_n(\mathbb{C}_m)$ with coefficients in the local system associated to the $B_n(\mathbb{C}_m)$-representation on $V^{\otimes n}\otimes W^{\otimes m}$ as the (co)homology of bimodules $\mathfrak{M}$ defined over the quantum shuffle algebra $\mathfrak{A} = \mathfrak{A}(V_\epsilon)$, where $V_\epsilon$ is the braided vector space $V$ with the braiding twisted by a sign.

\begin{thm}
Let $\mathcal{F}$ be the local system over $\mathrm{Conf}_n(\mathbb{C}_m)$ associated with the $B_n(\mathbb{C}_m)$-representation on $V^{\otimes n} \otimes W^{\otimes m}$, $\mathfrak{A}^e = \mathfrak{A} \otimes \mathfrak{A}^{op}$ be the enveloping algebra of $\mathfrak{A}$, and $\mathfrak{M}_i = \mathfrak{M}$ for all $1 \le i \le m$. Then there is an isomorphism
\[ H_*(\mathrm{Conf}_n(\mathbb{C}_m) \cup \{ \infty \}, \{ \infty \}; \mathcal{F}) \cong \left\{ \left( \left( \dots\left( \left(\mathfrak{M}_1 \underset{\mathfrak{A}}{\overset{L}{\otimes}} \mathfrak{M}_2 \right) \underset{\mathfrak{A}}{\overset{L}{\otimes}} \mathfrak{M}_3 \right) \dots \right) \underset{\mathfrak{A}}{\overset{L}{\otimes}} \mathfrak{M}_m \right) \underset{\mathfrak{A}^e}{\overset{L}{\otimes}} \mathbf{k} \right\}[-n][n+m] \]
where the first bracket denotes the shift in homological degree and the second the internal degree part.
\end{thm}

This result is proved as Corollary~\ref{cor:tor_n_times}. Section~\ref{sec:H*_conf} concludes with the main application of this theorem in the case relevant to character sums of the resultant: when $V=W=\mathbb{C}$ with the braidings given by multiplication by certain units, the $B_{n,m}$-representation on $V^{\otimes n} \otimes W^{\otimes m}$ defines the rank-$1$ local system $\mathcal{R}^*\mathcal{L}_\chi$ on $\mathrm{Conf}_{n,m}(\mathbb{C})$. In this case, we completely compute the homology of $\mathrm{Conf}_n(\mathbb{C}_m)$ (Theorems~\ref{thm:H*_punct_p=-q} and~\ref{thm:H*_punct_p_not_power}) and give a vanishing range for the homology of $\mathrm{Conf}_{n,m}(\mathbb{C})$ (Corollaries~\ref{cor:H*_mix_p=-q} and~\ref{cor:H*_mix_p_not_power}). Section~\ref{sec:char_sum} then completes the proof of Theorem~\ref{thm:main_intro} by translating our topological results to number theory using the machinery sketched at the beginning of the outline.


\subsection{Acknowledgments}

The author is indebted to Craig Westerland for numerous insights and valuable conversations. We owe many thanks to Jordan Ellenberg and Mark Shusterman for suggesting the paper's topological approach to study character sums of the resultant. We thank Calista Bernard, Hyun Jong Kim, Aaron Landesman, Jeremy Miller, Andrew Putman, Philip Tosteson, and Jesse Wolfson for various helpful discussions.


\section{Stratification of configuration spaces}\label{sec:fnf}

For any positive integer $m$, define $\mathbb{C}_m := \mathbb{C} \backslash \{ z_1, \dots, z_m \}$ to be the complex plane punctured by $m$ distinct points $z_1, \dots, z_m \in \mathbb{C}$. Without lost of generality, assume $z_k = k-1$ for all integers $1 \le k \le m$, i.e., the punctures in $\mathbb{C}_m$ are placed at the first $m$ non-negative integers on the real axis; in particular, $\mathbb{C}_1 = \mathbb{C}^\times$ is the complex plane punctured at the origin. In this section, we will construct a cellular stratification for configuration spaces of the $m$-punctured complex plane for any $m \ge 1$. We will first review the classical Fox-Neuwirth cellular stratification for $\mathrm{Conf}_n(\mathbb{C})$ (case $m = 0$) established by \cite{fn62,fuk70}. Our construction will generalize a stratification of $\mathrm{Conf}_n(\mathbb{C}_1)$ demonstrated by \cite{h22}.


\subsection{Classical Fox-Neuwirth cellular stratification}

We first recall the Fox-Neuwirth/Fuks stratification of Conf$_n(\mathbb{C})$ by Euclidean spaces, which provides a CW-complex structure for the 1-point compactification of Conf$_n (\mathbb{C})$. This construction was first demonstrated by \cite{fn62,fuk70} and further studied by \cite{vas92,gs12,etw17}; the treatment detailed here is summarized from \cite{etw17}.

A \textit{composition} $\lambda$ of $n$ (denoted by $\lambda \vdash n$) is an ordered partition $\lambda = (\lambda_1, \dots, \lambda_k)$ of $n$ where $\lambda_i > 0$ for all $1\le i\le k$ and $\sum \lambda_i = n$. The number of parts $k$ is called the \textit{length} of $\lambda$, denoted by $l(\lambda)$. Recall that the $n^\mathrm{th}$ symmetric product Sym$_n(\mathbb{R})$ has a stratification given by these partitions Sym$_n (\mathbb{R}) = \coprod_{\lambda \vdash n} \text{Sym}_\lambda (\mathbb{R})$, where elements of Sym$_\lambda (\mathbb{R})$ are unordered subsets of $l(\lambda)$ distinct points $x_1, \dots, x_{l(\lambda)}$, with the multiplicity of $x_i$ being $\lambda_i$.

A positive dimension cell Conf$_\lambda(\mathbb{C})$ (indexed by a composition $\lambda$ of $n$) of dimension $n+l(\lambda)$ is defined to be the preimage of Sym$_\lambda(\mathbb{R})$ under the projection map $\pi : \text{Conf}_n(\mathbb{C}) \to \text{Sym}_n(\mathbb{R})$ that projects each coordinate onto the real line, i.e., $\pi (x_1, \dots, x_n) = (\text{Re}(x_1), \dots, \text{Re}(x_n))$.
Loosely speaking, given a composition $\lambda$ of $n$, the cell Conf$_\lambda(\mathbb{C})$ consists of all configurations in $\mathrm{Conf}_n(\mathbb{C})$ where points are arranged into $l(\lambda)$ vertical columns and there are $\lambda_i$ points on the $i^\mathrm{th}$ column (from the left) for all $1 \le i \le l(\lambda)$ (see Figure~\ref{fig:FNF}). The \textit{lexicographic order} of points in such a configuration is defined by labelling the lowest point on the left most column with $1$, increasing the indices as we move up, and continuing the process for all subsequent columns on the right. The index of a point in this order is called the \textit{overall position} of that point in the configuration.

\begin{figure}[t]
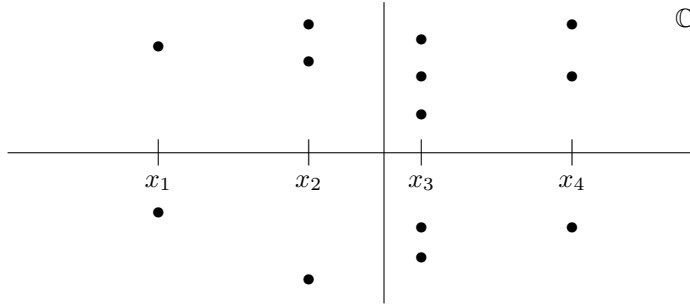

\begin{centering}
\[ \vcenter{	
	\xy
		(-30,14)*{\bullet}; (-30,-8)*{\bullet}; (-30,-4)*{x_1}; (-30,0)*{|};
		(-10,12)*{\bullet}; (-10,-17)*{\bullet}; (-10,17)*{\bullet}; (-10,-4)*{x_2}; (-10,0)*+{|};
		(5,10)*{\bullet}; (5,-10)*{\bullet}; (5,15)*{\bullet}; (5,-14)*{\bullet}; (5,5)*{\bullet}; (5,-4)*{x_3}; (5,0)*+{|};
		(25,10)*{\bullet}; (25,-10)*{\bullet}; (25,17)*{\bullet}; (25,-4)*{x_4}; (25,0)*+{|};
		{\ar@{-} (-50,0)*{}; (50,0)*{}};
		{\ar@{-} (0,-20)*{}; (0,20)*{}};
		(40,18)*{ \mathbb{C}};
	\endxy
	}
\]
	\caption{A configuration in Conf$_{(2,3,5,3)} (\mathbb{C}) \subset$ Conf$_{13} (\mathbb{C})$. The points in the configuration are arranged into four columns based on their (ordered) real coordinates, and the number of points on each column (starting from the left) is specified by the composition $(2,3,5,3)$.}
	\label{fig:FNF}
\end{centering}
\end{figure}

The boundary of a cell is obtained in two ways. The first type of boundary occurs by moving a point in a configuration to approach either the point at infinity or another point on the same vertical line; in this case, the boundary is the point at infinity. The second type of boundary occurs by horizontally joining two adjacent vertical columns of the configuration without colliding the points. The boundary cell of Conf$_\lambda (\mathbb{C})$ obtained by joining the $i^\mathrm{th}$ and $i+1^\mathrm{st}$ columns has the form Conf$_{\rho^i}(\mathbb{C})$, where $\rho^i = (\lambda_1, \dots, \lambda_{i-1}, \lambda_i + \lambda_{i+1}, \dots, \lambda_k)$ is the \textit{coarsening} of $\lambda$ obtained by summing $\lambda_i$ and $\lambda_{i+1}$ $(1 \le i < l(\lambda)).$

\begin{prop}[\cite{fn62,fuk70}]
The space $\mathrm{Conf}_n (\mathbb{C}) \cup \{ \infty \}$ has a CW-complex decomposition where the positive dimension cells are given by $\mathrm{Conf}_\lambda(\mathbb{C})$ (of dimension $n+l(\lambda)$) with indices $\lambda$ coming from compositions of $n$. The boundary of $\mathrm{Conf}_\lambda (\mathbb{C})$ is the union of $\mathrm{Conf}_\rho(\mathbb{C})$ where $\rho$ is a coarsening of $\lambda$. 
\end{prop}


\subsection{Cellular stratification of $\mathrm{Conf}_n (\mathbb{C}_m)$}\label{subsec:fnf_Cm}

For any integer $m \ge 1$, observe that there is a canonical embedding $\mathrm{Conf}_n (\mathbb{C}_m) \hookrightarrow \mathrm{Conf}_{n+m}(\mathbb{C})$ by inserting the previously removed points $z_1, \dots, z_m$. This gives a homeomorphic image of Conf$_n(\mathbb{C}_m)$ as a subspace of Conf$_{n+m} (\mathbb{C})$ consisting of all configurations where the points $z_1, \dots, z_m$ are always fixed. We will give a stratification of this subspace based on the Fox-Neuwirth cellular stratification of Conf$_{n+m} (\mathbb{C})$ introduced above. For the rest of this paper, we will indiscriminately use the notation Conf$_n (\mathbb{C}_m)$ for both the original configuration space of the $m$-punctured complex plane and its homeomorphic image embedded in Conf$_{n+m} (\mathbb{C})$.

Given a composition $\lambda$ of $n+m$, we consider the intersection of the cell Conf$_{\lambda}(\mathbb{C})$ and the subspace Conf$_n (\mathbb{C}_m)$ of Conf$_{n+m}(\mathbb{C})$. Starting with a configuration in Conf$_{\lambda}(\mathbb{C})$, the restriction on the fixed points $z_1, \dots, z_m$ results in two constraints. First, for all $1 \le k \le m$, the vertical column that contains $z_k$ (indexed by $i_k$) must be fixed, i.e., the real part of all points on that column must be $z_k = k-1$; we refer to these columns as the \textit{fixed columns} in the cell, and others the \textit{free columns}. Since the fixed points all have distinct real parts which keeps the fixed columns separate, the number of vertical columns in a cell must be at least $m$. Secondly, for every $1 \le k \le m$, it is forbidden for points on the $k^\mathrm{th}$ fixed column ($i_k^\mathrm{th}$ overall) to move past the fixed point $z_k$. The number of points on this vertical line with a negative imaginary part is hence fixed and denoted by the index $j_k$. Therefore, the connected components in the above intersection can be denoted by $e_{(\lambda, I, J)} = \text{Conf}_{(\lambda, I, J)} (\mathbb{C})$ where $\lambda$ is a composition of $n+m$, $I = (i_1,\dots,i_m)$ is the $m$-tuple of indices of the fixed columns ($1 \le i_1 < \dots < i_m \le l(\lambda)$), and $J = ( j_1, \dots, j_m )$ where $j_k$ is the number of points lying below $z_k$ on the $i_k^\mathrm{th}$ column ($0 \le j_k \le \lambda_{i_k} - 1$). Given a composition $\lambda$, the overall position $\iota$ of a point $z$ in a cell indexed by $\lambda$ is the same information as the pair of indices $(i,j)$ where $i$ indexes the column containing $z$ and $j$ indexes the number of points lying below $z$ on that column, via the identification $\iota = j+1+\sum_{k=1}^{i-1} \lambda_k$. Hence it is possible to re-index the subspace $e_{(\lambda,I,J)}$ by the composition $\lambda$ and a tuple $\mathcal{I} = (\iota_1,\dots,\iota_m)$ that contains the overall positions of the fixed points $z_1,\dots,z_m$ (see Figure~\ref{fig:FNF_punctured}).

\begin{figure}[t]
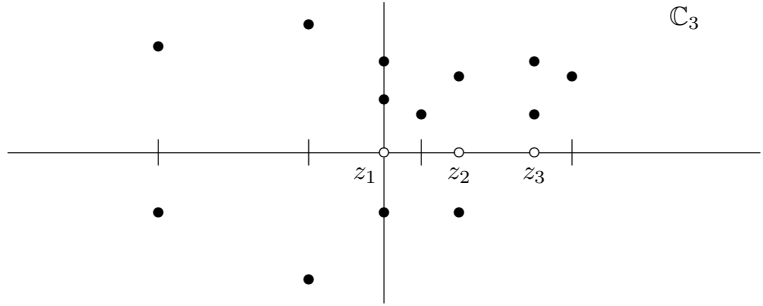

\begin{centering}
\[ \vcenter{	
	\xy
		(-30,14)*{\bullet}; (-30,-8)*{\bullet}; (-30,0)*{|};
		(-10,-17)*{\bullet}; (-10,17)*{\bullet}; (-10,0)*+{|};
		(0,7)*{\bullet}; (0,12)*{\bullet}; (0,-8)*{\bullet}; (0,-0.08)*{\circ}; (-2.5,-3)*{z_1};
		(5,5)*{\bullet}; (5,0)*+{|};
		(10,10)*{\bullet}; (10,-8)*{\bullet}; (10,-0.08)*{\circ}; (10,-3)*{z_2};
		(20,5)*{\bullet}; (20,12)*{\bullet}; (20,-0.08)*{\circ}; (20,-3)*{z_3};
		(25,10)*{\bullet}; (25,0)*+{|};
		{\ar@{-} (-50,0)*{}; (-0.65,0)*{}}; {\ar@{-} (0.65,0)*{}; (9.35,0)*{}}; {\ar@{-} (10.65,0)*{}; (19.35,0)*{}}; {\ar@{-} (20.65,0)*{}; (50,0)*{}};
		{\ar@{-} (0,-20)*{}; (0,-0.65)*{}}; {\ar@{-} (0,0.65)*{}; (0,20)*{}};
		(40,18)*{ \mathbb{C}_{3}};
	\endxy
	}
\]
	\caption{A configuration in $e_{(\lambda,I,J)} \subset \mathrm{Conf}_{13} (\mathbb{C}_3)$, where $\lambda = (2,2,4,1,3,3,1)$, $I = (3, 5, 6)$, and $J = (1,1,0)$. Alternatively, this cell can be indexed by $\lambda$ and the tuple $\mathcal{I} = (6,11,13)$.}
	\label{fig:FNF_punctured}
\end{centering}
\end{figure}

The spaces $e_{(\lambda,I,J)}$ then provide the positive dimension cells for our cellular decomposition of Conf$_n(\mathbb{C}_m) \cup \{ \infty \}$. Each cell $e_{(\lambda,I,J)}$ is homeomorphic to the product
\[\begin{array}{l}
\left[ \mathrm{Conf}_{i_1-1}((-\infty,0)) \times \left( \displaystyle \prod^{m-1}_{k=1} \mathrm{Conf}_{i_{k+1} - i_k -1}((k-1,k)) \right)
\times \mathrm{Conf}_{l(\lambda)-i_m}((m-1,\infty)) \right] \\[15pt]
\quad \times \displaystyle \prod^{l(\lambda)}_{i=1,i\not\in I} \mathrm{Conf}_{\lambda_i} (\mathbb{R}) \times \left[ \displaystyle \prod^{m}_{k=1} \mathrm{Conf}_{j_k} ((-\infty,0)) \times \mathrm{Conf}_{\lambda_{i_k} - j_k -1} ((0,\infty)) \right].
\end{array}\]
The first bracket represents the configurations of the free columns before the first fixed column, between each pair of fixed columns, and after the last fixed column, i.e., recording the real parts of the points. The last bracket keeps track of the imaginary parts of points below and above the fixed point on each fixed column, while the middle product records the same information for those on the free columns. By identifying each open interval with $\mathbb{R}$ and applying the homeomorphism Conf$_k(\mathbb{R}) \cong \mathbb{R}^k$, we see that the cell $e_{(\lambda,I,J)}$ has dimension $n+l(\lambda)-m$; loosely speaking, compared to the classical Fox-Neuwirth cell indexed by the same composition $\lambda$, we lost $2m$ dimensions due to fixing the real and imaginary parts of $m$ points.

As in the classical Fox-Neuwirth cellular decomposition of Conf$_n(\mathbb{C})$, the boundary of a cell in this stratification is obtained in two ways. For the first type, we let points in a configuration approach other (free or fixed) points or infinity; in this case, the boundary is still the point at infinity. The second type of boundary again occurs by horizontally joining two adjacent vertical columns of the configuration without colliding the points. Note that the fixed columns are not allowed to merge with one another. Due to the second constraint of the cell, namely points on a fixed column cannot move across the fixed point in that column, the boundary cells obtained this way have five general forms, depending on the types of columns (free or fixed) involved in the column combination and their relative positions. In particular, when combining a free column with the $k^\mathrm{th}$ fixed column, we must keep track of the number of points going below the fixed point, i.e., adding to the index $j_k$; in the alternate indexing system, this results in a change of the overall position of the fixed point $z_k$. In summary:

\begin{prop}
The space $\mathrm{Conf}_n(\mathbb{C}_m) \cup \{ \infty \}$ may be presented as a CW complex whose positive dimension cells $e_{(\lambda, I, J)} = \mathrm{Conf}_{(\lambda, I, J)}(\mathbb{C})$ (of dimension $n+l(\lambda)-m$) are indexed by triples $(\lambda,I,J)$, where $\lambda$ is an ordered partition of $n+m$, $I = (i_1,\dots,i_m)$ is the $m$-tuple of indices of the fixed columns $(1 \le i_1 < \dots < i_m \le l(\lambda))$, and $J = ( j_1, \dots, j_m )$ where $j_k$ is the number of points on the $i_k^\mathrm{th}$ column with negative imaginary parts $(0 \le j_k \le \lambda_{i_k} - 1)$.

Let $I_k$ denote the $m$-tuple $(i_1, \dots, i_{k-1}, i_k-1,\dots,i_m-1)$, and $J_{k,h}$ denote $(j_1,\dots,j_k+h,\dots,j_m)$. The codimension-$1$ boundary cells of $e_{(\lambda, I, J)}$ have five general forms:
\begin{enumerate}
\item $e_{(\rho^i, I_1, J)}$ \hfill $1 \le i < i_1-1$
\item $e_{(\rho^i, I_{k+1}, J)}$ \hfill $1 \le k \le m-1$, $i_k < i < i_{k+1}-1$
\item $e_{(\rho^i, I, J)}$ \hfill $i_k < i < l(\lambda)$
\item $e_{(\rho^{i_k-1}, I_k, J_{k,h})}$ \hfill $1 \le k \le m$, $0 \le h \le \lambda_{i_k-1}$
\item $e_{(\rho^{i_k}, I_{k+1}, J_{k,h})}$ \hfill $1 \le k \le m$, $0 \le h \le \lambda_{i_k+1}$
\end{enumerate}
where $\rho^i = (\lambda_1, \dots, \lambda_i + \lambda_{i+1}, \dots, \lambda_{l(\lambda)})$ is the coarsening of $\lambda$ obtained by summing $\lambda_i$ and $\lambda_{i+1}$ $(1 \le i < l(\lambda))$, and $h$ denotes the number of points going below the fixed point when combining a fixed column with the free column on its left (4) or right (5).
\end{prop}

The overall positions of all fixed points are unchanged in the first three types of codimension-1 boundaries, whereas only the position of $z_k$ changes to $\iota_k - \lambda_{i_k-1}+h$ (type $4$) or $\iota_k+h$ (type 5). In particular, the relative order of the overall positions of the fixed points is always preserved. This property is very crucial to the argument of this paper and will be revisited frequently in later sections.

Observe that while combining two adjacent columns of a configuration to obtain a boundary cell, we implicitly make a choice of shuffling the points into a single column. This is governed by a $(p,q)$-\textit{shuffle} $\gamma$, defined to be a bijection $\gamma : \{1,\dots,p\} \sqcup \{1,\dots,q\} \to \{1,\dots,p+q\}$ that preserves the order on both $\{1,\dots,p\}$ and $\{1,\dots,q\}$. Alternatively, a $(p,q)$-shuffle can be interpreted as a permutation on $p+q$ elements that preserves the order on the first $p$ and the last $q$ elements. A $(p,(q,h),j)$\textit{-shuffle} is defined to be a $(p,q)$-shuffle that (as a permutation on $p+q$ elements) sends $j+1$ to $j+h+1$. Similarly, a $((p,h),q,j)$\textit{-shuffle} is a $(p,q)$-shuffle that sends $p+j+1$ to $h+j+1$. These specific types of shuffles govern the point shuffling when combining a fixed column with a free column, and thus are crucial to the construction in this paper; for an analysis of their properties, see \cite[\S4.1]{h22}. We denote the set of $(p,q)$-shuffles by $\mathrm{Sh}(p,q)$.

For integers $p$ and $q$, let $c_{p,q} = \sum_\gamma (-1)^{|\gamma|}$ be the sum of the signs of all $(p,q)$-shuffles $\gamma$. Similarly, for integers $0 \le h \le p$ and $0 \le j < q$, let $c_{q,(p,h),j}$ and $c_{(p,h),q,j}$ be the sums of the signs of all $(q,(p,h),j)$- and $((p,h),q,j)$-shuffles, respectively. From the stratification, we can write down an explicit cellular chain complex for the 1-point compactification Conf$_n (\mathbb{C}_k) \cup \{ \infty \}$.

\begin{defn}[Fox-Neuwirth complex for Conf$_n (\mathbb{C}_k) \cup \{\infty\}$]
    Let $D(n,m)_*$ denote the chain complex which in degree $q$ is generated over $\mathbb{Z}$ by the set of triples $(\lambda,I,J)$ where $\lambda = (\lambda_1, \dots, \lambda_{q-n+m})$ is a composition of $n+m$ of length $q-n+m$, $I = (i_1,\dots,i_m)$ with $1 \le i_1 < \dots < i_m \le l(\lambda)$, and $J = (j_1,\dots,j_m)$ with $0 \le j_k \le \lambda_{i_k} -1$.
    
    The differential $d: D(n,m)_q \to D(n,m)_{q-1}$ is given by the formula
    \[\begin{split}
    \displaystyle d(\lambda, I, J) & = \sum_{i = 1}^{i_1-2} (-1)^{i-1} c_{\lambda_i, \lambda_{i+1}} (\rho^i, I_1, J)\\
    & + \sum_{k=1}^{m-1} \sum_{i = i_k+1}^{i_{k+1}-2} (-1)^{i-1} c_{\lambda_i, \lambda_{i+1}} (\rho^i, I_{k+1}, J)\\
    & + \sum_{i = i_m+1}^{q-n+m-1} (-1)^{i-1} c_{\lambda_i, \lambda_{i+1}} (\rho^i, I, J)\\
    & + \sum_{k=1}^m (-1)^{i_k-2} \displaystyle\sum_{h=0}^{\lambda_{i_k-1}} c_{(\lambda_{i_k-1},h),\lambda_{i_k}, j_k} (\rho^{i_k-1}, I_k, J_{k,h})\\
    & + \sum_{k=1}^m (-1)^{i_k-1}\displaystyle\sum_{h=0}^{\lambda_{i_k+1}} c_{\lambda_{i_k}, (\lambda_{i_k+1},h),j_k} (\rho^{i_k}, I_{k+1}, J_{k,h}).
    \end{split}\]
\end{defn}

The signs in the formula of the differential result from the induced orientations on the boundary strata, following the general framework described by \cite{gs12}. There is a simple formula to compute the constant $c_{p,q}$ using the quantum binomial coefficient (see, e.g., Proposition 1.7.1 of \cite{sta12}):
\[c_{p,q} = \displaystyle {p + q \choose q}_{-1}. \]
The constants $c_{q,(p,h),j}$ and $c_{(p,h),q,j}$ can be computed in terms of constants $c_{u,v}$ as a corollary of decomposition theorems of the $(q,(p,h),j)$- and $((p,h),q,j)$-shuffles (see Lemma 4.5 of \cite{h22}).

The proof that the chain complex $D(n,m)_*$ is well-defined rests on the following observation: all types of boundary cells in this chain complex arise in the exact same way independent of the number $m$ of fixed points. It follows that this proof can be reduced to the special case $m = 1$, which was demonstrated in Proposition 4.7 of \cite{h22}.

\begin{prop}
$d^2 = 0$.
\end{prop}

\begin{proof}
The general strategy is to enumerate all types of boundary cells in $d^2 (\lambda,I,J)$ and show that their coefficients all vanish. Loosely speaking, cells in $d^2 (\lambda,I,J)$ are formed by subsequently performing two column combinations in the cell $e_{(\lambda,I,J)}$. There are two outcomes: either (1) two pairs of columns in $e_{(\lambda,I,J)}$ are combined separately, or (2) three adjacent columns are combined into a single column. Within each of these types, the argument follows the same logic and only differs slightly in the exact details depending on the types of columns (free or fixed) involved in the combinations and their relative positions. Therefore, we will present the argument for a representative of each type.

For a representative of type (1), we consider the boundary cell
\[((\lambda_1, \dots, \lambda_{i_r-1}+\lambda_{i_r}, \dots, \lambda_i + \lambda_{i+1}, \dots, \lambda_{l(\lambda)}),(I_r)_{s+1},J_{r,h})\]
obtained by joining two pairs of columns indexed by $\{i_r-1,i_r\}$ and $\{i,i+1\}$ $(i_r < i_s < i < i_{s+1}-1)$. Since $J_{r,h}$ is completely determined by the combination of the first pair of columns and $(I_r)_{s+1} = (I_{s+1})_r$, both of the following orders to perform the operations result in this boundary cell: either (1a) combining the first pair then the second pair, or (1b) combining the second pair first. Thus the coefficient of the cell above in $d^2(\lambda,I,J)$ is
\[\begin{array}{c}
(-1)^{i_r-2} c_{(\lambda_{i_r-1},h),\lambda_{i_r}, j_r} \cdot (-1)^{i-2}c_{\lambda_i,\lambda_{i+1}}\\[5pt]
+ (-1)^{i-1}c_{\lambda_i,\lambda_{i+1}} \cdot (-1)^{i_r-2}  c_{(\lambda_{i_r-1},h),\lambda_{i_r}, j_r} = 0.\end{array}\]

For a representative of type (2), consider a boundary cell of the form
\[((\lambda_1, \dots, \lambda_{i_r-2}+\lambda_{i_r-1}+\lambda_{i_r}, \dots, \lambda_{l(\lambda)}),(I_r)_r,J_{r,h})\]
obtained by joining three adjacent columns indexed by $\{i_r-2,i_r-1,i_r\}$. Similarly, there are two orders to perform the operations: either (2a) combining the first two columns then combining the joint column with the third, or (2b) combining the last two columns first. Particularly in case (2b), the fixed column involves in both column combinations, so the $h$ points that move below the fixed point $z_r$ in the final configuration can be split into two steps: $s$ points in the first operation followed by $h-s$ points in the second ($0\le s\le h$). The coefficient of this representative cell in $d^2(\lambda,I,J)$ hence contains a sum over all $s$:
\[\begin{array}{c}
(-1)^{i_r-3} c_{\lambda_{i_r-2},\lambda_{i_r-1}} \cdot (-1)^{i_r-3} c_{(\lambda_{i_r-2}+\lambda_{i_r-1},h),\lambda_{i_r}, j_r} \\[5pt]
+ \displaystyle \sum_{s=0}^h (-1)^{i_r-2} c_{(\lambda_{i_r-1},s),\lambda_{i_r}, j_r} \cdot (-1)^{i_r-3}  c_{(\lambda_{i_r-2},h-s),\lambda_{i_r-1} + \lambda_{i_r}, j_r+s}.\end{array}\]
This expression can be shown to vanish using several identities of the $(p,q)$- and $((p,h),q,j)$-shuffles; a proof of this fact can be found in the proof of Proposition 4.7 of \cite{h22}.
\end{proof}

By construction, the complex $D(n,m)_*$ is isomorphic to the relative cellular chain complex of Conf$_n (\mathbb{C}_m) \cup \{ \infty \}$, relative to the point at infinity. In particular,
\[H_*(\mathrm{Conf}_n(\mathbb{C}_m) \cup \{\infty\}, \{\infty\}) \cong H_*(D(n,m)_*).\]
This construction therefore provides an approach to compute the cellular homology of configuration spaces of multi-punctured complex planes with trivial coefficients.


\section{Representations of braid subgroups}\label{sec:rep_br}

In this section, we will explore the algebra and representation theory of braid subgroups that are related to several configuration spaces of planes, namely $\mathrm{Conf}_n(\mathbb{C}_m)$ and bicolor configuration spaces $\mathrm{Conf}_{n,m}(\mathbb{C})$. We will first study the fundamental groups of these spaces and develop a theory of induced representation for these groups.


\subsection{Fundamental groups of configuration spaces}\label{subsec:br_subgps}

Recall that the $n^{\mathrm{th}}$ \textit{ordered configuration space of} $\mathbb{C}$ is defined to be the space
\[\mathrm{PConf}_n(\mathbb{C}) := \{(c_1,\dots,c_n) \in \mathbb{C}^n : c_i \ne c_j \text{ if } i \ne j \}.\]
The fundamental group of $\mathrm{PConf}_n(\mathbb{C})$ is the classical \textit{pure braid group} $PB_n$ on $n$ strands. There is a natural action of the symmetric group $S_n$ on $\mathrm{PConf}_n(\mathbb{C})$ by permuting the coordinates; the quotient space $\mathrm{PConf}_n(\mathbb{C})/S_n$ can be identified with the unordered configuration space $\mathrm{Conf}_n(\mathbb{C})$ as previously defined. The fundamental group of $\mathrm{Conf}_n(\mathbb{C})$ is the classical \textit{Artin's braid group} $B_n$, which may be presented as
\[B_n = \langle \sigma_1, \dots, \sigma_{n-1} : \sigma_i \sigma_j = \sigma_j \sigma_i \text{ if } |i-j| > 1; \sigma_i \sigma_{i+1} \sigma_i = \sigma_{i+1} \sigma_i \sigma_{i+1} \rangle.\]
There is an exact sequence of groups
\[ 1 \to PB_n \to B_n \to S_n \to 1,\]
where the map $B_n \to S_n$ sends a braid $b$ to its underlying permutation $\underline{b}$ \cite{fn62}. 
The pure braid group $PB_n$ can be presented as a subgroup of $B_n$ generated by the elements
\[ \theta_{ij} = \sigma^{-1}_i \cdots \sigma^{-1}_{j-2} \sigma^2_{j-1} \sigma_{j-2} \cdots \sigma_i \]
for $1 \le i < j \le n$ \cite{bir74}\footnote{Strictly speaking, our formula of $\theta_{ij}$ represents a mirror image of the pure braid generators given by \cite{bir74}. This alternate choice of generators proves to be more compatible with our topological constructions in this paper.}; geometrically, the generator $\theta_{ij}$ is represented the braid that wraps the $i^\mathrm{th}$ strand around the $j^\mathrm{th}$ strand (see Figure~\ref{fig:PB_gen}).

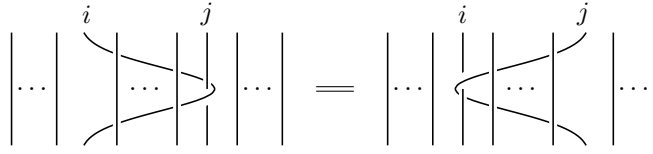
\begin{figure}[t]
    \centering
\begin{tikzpicture}
\draw[white,double=black,very thick,-] (0,0) -- (0,1.5);
\draw[white,double=black,very thick,-] (0.6,0) -- (0.6,1.5);
\node at (0.3,0.75) {\dots};
\draw[white,double=black,very thick,-] (3,0) -- (3,1.5);
\draw[white,double=black,very thick,-] (3.6,0) -- (3.6,1.5);
\node at (3.3,0.75) {\dots};
\draw[white,double=black,very thick,-] (2.6,0) -- (2.6,0.75);
\node at (1,1.75) {$i$};
\draw[smooth,white,double=black,line width=1mm,-] plot[variable=\x,domain=0:1.5] ({1.8*exp(-6*(\x-0.75)*(\x-0.75))+0.9},{\x});
\node at (2.6,1.75) {$j$};
\draw[white,double=black,very thick,-] (2.6,0.75) -- (2.6,1.5);
\draw[white,double=black,very thick,-] (1.4,0) -- (1.4,1.5);
\node at (1.8,0.75) {\dots};
\draw[white,double=black,very thick,-] (2.2,0) -- (2.2,1.5);
\draw[-] (4.05,0.72) -- (4.55,0.72);
\draw[-] (4.05,0.78) -- (4.55,0.78);
\draw[white,double=black,very thick,-] (5,0) -- (5,1.5);
\draw[white,double=black,very thick,-] (5.6,0) -- (5.6,1.5);
\node at (5.3,0.75) {\dots};
\draw[white,double=black,very thick,-] (8,0) -- (8,1.5);
\draw[white,double=black,very thick,-] (8.6,0) -- (8.6,1.5);
\node at (8.3,0.75) {\dots};
\node at (6.0,1.75) {$i$};
\draw[white,double=black,very thick,-] (6.0,0.75) -- (6.0,1.5);
\node at (7.6,1.75) {$j$};
\draw[smooth,white,double=black,line width=1mm,-] plot[variable=\x,domain=0:1.5] ({-1.8*exp(-6*(\x-0.75)*(\x-0.75))+7.7},{\x});
\draw[white,double=black,very thick,-] (6.0,0) -- (6.0,0.75);
\draw[white,double=black,very thick,-] (6.4,0) -- (6.4,1.5);
\node at (6.8,0.75) {\dots};
\draw[white,double=black,very thick,-] (7.2,0) -- (7.2,1.5);
\end{tikzpicture}
    \caption{Pure braid generator $\theta_{ij}$.}
    \label{fig:PB_gen}
\end{figure}

Observe that there is an action of the group $S_n \times S_m$ as a subgroup of $S_{n+m}$ on $\mathrm{PConf}_{n+m}(\mathbb{C})$. The quotient space $\mathrm{PConf}_{n+m}(\mathbb{C})/(S_n \times S_m)$ can be identified with the \textit{bicolor configuration space} $\mathrm{Conf}_{n,m}(\mathbb{C})$ consisting of all configurations of $n$ blue points and $m$ red points in $\mathbb{C}$; its fundamental group is the $(n,m)$\textit{-mixed braid group} $B_{n,m}$, the subgroup of $B_{n+m}$ containing braids with $n$ blue strands and $m$ red strands that preserve the partition $(n,m)$ on the endpoints. Equivalently, $B_{n,m}$ is the preimage of $S_n \times S_m$ under the map $B_{n+m} \to S_{n+m}$. Manfredini gave a presentation of this group in \cite{man97}, in terms of the braid generators $\{\sigma_i\}_{i\ne n}$ of $B_{n+m}$ and $\tau_n = \sigma^2_n$, with the relations
\[\begin{array}{l}
     \sigma_i \sigma_j = \sigma_j \sigma_i \text{ if } |i-j| > 1; \\
     \sigma_i \tau_n = \tau_n \sigma_i \text{ and } \sigma_i \sigma_{i+1} \sigma_i = \sigma_{i+1} \sigma_i \sigma_{i+1} \text{ if } i < n-1 \text{ or } i > n; \\
     \sigma_{i} \tau_n \sigma_{i} \tau_n = \tau_n \sigma_{i} \tau_n \sigma_{i} \text{ if } i = n-1, n+1; \text{ and }\\
     \sigma_{n-1}\tau_n\sigma^{-1}_{n-1}\sigma_{n+1}\tau_n\sigma^{-1}_{n+1} = \sigma_{n+1}\tau_n\sigma_{n+1}^{-1}\sigma_{n-1}\tau_n\sigma_{n-1}^{-1}.
\end{array}\]
Pictorially, the generators $\sigma_{i<n}$, $\sigma_{i>n}$, and $\tau_n$ are represented by the crossings of two blue strands, two red strands, and the full twist of the $n^\mathrm{th}$ strand (blue) and the $n+1^\mathrm{st}$ strand (red), respectively (see Figure~\ref{fig:Bnm_gen}).

\begin{figure}[t]
    \centering
\begin{tikzpicture}
\draw[white,double=blue,very thick,-] (0,0) -- (0,1.5);
\draw[white,double=blue,very thick,-] (0.6,0) -- (0.6,1.5);
\node at (0.3,0.75) {\dots};
\draw[white,double=blue,very thick,-] (1.2,0) -- (0.9,1.5);
\draw[white,double=blue,very thick,-] (0.9,0) -- (1.2,1.5);
\node at (0.9,1.75) {$i$};
\draw[white,double=blue,very thick,-] (1.5,0) -- (1.5,1.5);
\draw[white,double=blue,very thick,-] (2.1,0) -- (2.1,1.5);
\node at (1.8,0.75) {\dots};
\draw[white,double=red,very thick,-] (2.4,0) -- (2.4,1.5);
\draw[white,double=red,very thick,-] (3,0) -- (3,1.5);
\node at (2.7,0.75) {\dots};
\draw[white,double=blue,very thick,-] (4,0) -- (4,1.5);
\draw[white,double=blue,very thick,-] (4.6,0) -- (4.6,1.5);
\node at (4.3,0.75) {\dots};
\draw[white,double=red,very thick,-] (4.9,0) -- (4.9,1.5);
\draw[white,double=red,very thick,-] (5.5,0) -- (5.5,1.5);
\node at (5.2,0.75) {\dots};
\draw[white,double=red,very thick,-] (6.1,0) -- (5.8,1.5);
\draw[white,double=red,very thick,-] (5.8,0) -- (6.1,1.5);
\node at (5.8,1.75) {$i$};
\draw[white,double=red,very thick,-] (6.4,0) -- (6.4,1.5);
\draw[white,double=red,very thick,-] (7,0) -- (7,1.5);
\node at (6.7,0.75) {\dots};
\draw[white,double=blue,very thick,-] (8,0) -- (8,1.5);
\draw[white,double=blue,very thick,-] (8.6,0) -- (8.6,1.5);
\node at (8.3,0.75) {\dots};
\draw[white,double=red,very thick,-] (9.2,0) -- (9.2,0.75);
\node at (8.9,1.75) {$n$};
\draw[smooth,white,double=blue,line width=1mm,-] plot[variable=\x,domain=0:1.5] ({0.4*exp(-6*(\x-0.75)*(\x-0.75))+8.88},{\x});
\draw[white,double=red,very thick,-] (9.2,0.75) -- (9.2,1.5);
\draw[white,double=red,very thick,-] (9.5,0) -- (9.5,1.5);
\draw[white,double=red,very thick,-] (10.1,0) -- (10.1,1.5);
\node at (9.8,0.75) {\dots};
\end{tikzpicture}
    \caption{Mixed braid generators $\sigma_{i<n}$, $\sigma_{i>n}$, and $\tau_n$.}
    \label{fig:Bnm_gen}
\end{figure}
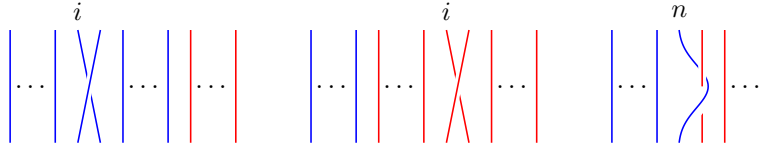

In \cite{fadneu62}, Fadell and Neuwirth showed that there is a fibration
\[ \mathrm{Conf}_n(\mathbb{C}_m) \to \mathrm{Conf}_{n,m}(\mathbb{C}) \to \mathrm{Conf}_m (\mathbb{C}) \]
that results in an exact sequence of fundamental groups
\[ 1 \to B_n(\mathbb{C}_m) \to B_{n,m} \to B_m \to 1.\]
where $B_n(\mathbb{C}_m)$ is the {\it surface braid group on $n$ strands} on $\mathbb{C}_m \cong \Sigma_{0,m+1}$, a surface of genus $0$ with $m+1$ punctures. This group is isomorphic to the subgroup of $B_{n+m}$ consisting of braids whose last $m$ strands are straight, or equivalently, $(n,m)$-mixed braids with only straight red strands. For $k = 1$, $B_n(\mathbb{C}_1)$ is isomorphic to the $(n,1)$-mixed braid group $B_{n,1}$, due to the fact that it is always possible to ``straighten'' the last pure strand. Bellingeri and Godelle gave a positive presentation of all surface braid groups \cite{bellingeri-godelle}; in particular, $B_n(\mathbb{C}_m)$ can be generated by the generators $\sigma_1, \dots, \sigma_{n-1}$ of $B_{n+m}$, and the generators $\theta_{nj}$ for $n+1 \le j \le n+m$ of the pure braid group $PB_{n+m}$.
We may rewrite the pure braid generators as $\theta_{ij} = \sigma_i^{-1} \dots \sigma_{j-2}^{-1} \sigma_{j-1}^2 \sigma_{j-2}\dots\sigma_i = \sigma_{j-1} \dots \sigma_{i+1} \sigma^2_i \sigma_{i+1}^{-1} \dots \sigma_{j-1}^{-1}$ (see Figure~\ref{fig:PB_gen}). It follows that there is a natural embedding of $B_n(\mathbb{C}_m)$ into $B_{n,m}$ that sends the generators $\sigma_i$ to the corresponding $\sigma_i$ in $B_{n,m}$ for $1 \le i \le n-1$, and $\theta_{nj}$ to $\sigma_{j-1} \dots \sigma_{n+1} \tau_n \sigma_{n+1}^{-1} \dots \sigma_{j-1}^{-1}$ for $n+1 \le j \le n+m$.

\subsection{Mixed-braided vector spaces}\label{subsec:mix_bvs}

We develop a family of representations of mixed braid groups coming from braided vector spaces. Let $\mathbf{k}$ be a field; unless otherwise noted, all tensor products will be over $\mathbf{k}$.

\begin{defn}
    A \textit{braided vector space} $V$ over $\mathbf{k}$ is a finite dimensional $\mathbf{k}$-vector space equipped with an invertible \textit{braiding} $\sigma : V \otimes V \to V \otimes V$ such that it satisfies the braid equation on $V^{\otimes 3}$:
\[(\sigma \otimes \mathrm{id}) \circ (\mathrm{id} \otimes \sigma) \circ (\sigma \otimes \mathrm{id}) = (\mathrm{id} \otimes \sigma) \circ (\sigma \otimes \mathrm{id}) \circ (\mathrm{id} \otimes \sigma).\]
\end{defn}

There is a natural action of the braid group $B_n$ on $V^{\otimes n}$ defined by $\sigma_i \mapsto \mathrm{id}^{\otimes i-1} \otimes \sigma \otimes \mathrm{id}^{\otimes n-i-1}$. We will write elements of $V^{\otimes n}$ using bar complex notation, i.e., $[a_1|\dots|a_n]$. Let $(V,\sigma)$ be a braided vector space.

\begin{defn}\label{defn:qsa}
The \textit{quantum shuffle algebra} $\mathfrak{A}(V)$ is a braided, graded bialgebra: its underlying coalgebra is the tensor coalgebra
\[ T^{co}(V) = \displaystyle \bigoplus_{n \ge 0} V^{\otimes n}\]
with the deconcatenation coproduct $\Delta$, equipped with a multiplication given by the quantum shuffle product:
\[\displaystyle [a_1 | \dots | a_n] \star [b_1 | \dots | b_m] = \sum_{\gamma} \widetilde{\gamma} [a_1 | \dots | a_n | b_1 | \dots | b_m]\]
where the sum is over all $(n,m)$-shuffles $\gamma$, and $\widetilde{\gamma} \in B_{n+m}$ is the lift of $\gamma$ obtained by moving the right $m$ strands in front of the left $n$ strands (see Figure~\ref{fig:shuf_lift}).
\end{defn}

\begin{figure}[t]
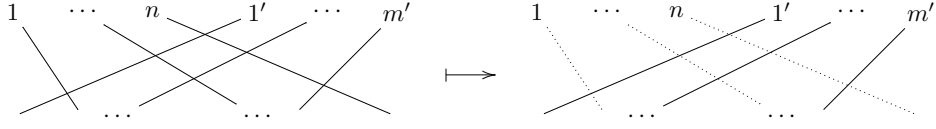

\begin{centering}
    \[ 
\resizebox{\textwidth}{!}{
	\xy
		{\ar@{-} (0,0)*+{}; (35,15)*+{1'} };
		{\ar@{-} (15,0)*+{\cdots}; (45,15)*+{\cdots} }; 
		{\ar@{-} (40,0)*+{}; (55,15)*+{m'} }; 
		{\ar@{-} (10,0)*+{}; (0,15)*+{1} }; 
		{\ar@{-} (35,0)*+{\cdots}; (10,15)*+{\cdots} }; 
		{\ar@{-} (55,0)*+{}; (20,15)*+{n} }; 
		{\ar@{|->} (62,6)*{}; (69,6)*{}};
		{\ar@{-} (75,0)*+{}; (110,15)*+{1'} };
		{\ar@{-} (90,0)*+{\cdots}; (120,15)*+{\cdots} }; 
		{\ar@{-} (115,0)*+{}; (130,15)*+{m'} }; 
		{\ar@{..} (85,0)*+{}; (75,15)*+{1} }; 
		{\ar@{..} (110,0)*+{\cdots}; (85,15)*+{\cdots} }; 
		{\ar@{..} (130,0)*+{}; (95,15)*+{n} }; 
	\endxy
	}
\]
    \captionsetup{justification=centering}
    \caption{Lifting an $(n,m)$-shuffle to a braid.}
    \label{fig:shuf_lift}
\end{centering}
\end{figure}

The quantum shuffle algebra has the structure of a Hopf algebra in a braided monoidal category. Relevant to the approach of this paper, Ellenberg, Tran, and Westerland recently identified the homology of the braid group $B_n$ with coefficients in $V^{\otimes n}$ with the cohomology of the quantum shuffle algebra $\mathfrak{A} = \mathfrak{A}(V_\epsilon)$, where $V_\epsilon$ is $V$ with the braiding twisted by a sign \cite{etw17}.

In \cite{h22}, we defined an analogue of braided vector spaces that is tailored for computations related to the mixed braid group $B_{n,1}$.

\begin{defn}
A \textit{left-braided vector space} $(V, W)$ over $\mathbf{k}$ is a pair of finite dimensional $\mathbf{k}$-vector spaces $V$ and $W$, where $V$ is a braided vector space with a braiding $\sigma$, further equipped with another isomorphism $\tau : V \otimes W \to V \otimes W$ such that it satisfies an additional braid equation on $V^{\otimes 2} \otimes W$:
\[(\sigma \otimes \mathrm{id}) \circ (\mathrm{id} \otimes \tau) \circ (\sigma \otimes \mathrm{id}) \circ (\mathrm{id} \otimes \tau)  = (\mathrm{id} \otimes \tau) \circ (\sigma \otimes \mathrm{id}) \circ (\mathrm{id} \otimes \tau) \circ (\sigma \otimes \mathrm{id}).\]
\end{defn}

Left-braided vector spaces form a category where morphisms $(V_1,W_1,\sigma_1,\tau_1) \to (V_2,W_2,\sigma_2,\tau_2)$ are pairs of $\mathbf{k}$-linear maps $f_V :V_1 \to V_2$ and $f_W : W_1\to W_2$, such that $f_V$ is a morphism of braided vector spaces $(V_1,\sigma_1) \to (V_2, \sigma_2)$ and $f_W$ satisfies the relation: $(f_V \otimes f_W) \circ \tau_1 = \tau_2 \circ (f_V \otimes f_W)$ on $V_1 \otimes W_1$.
Analogous to braided vector spaces, there is an action of $B_{n,1}$ on $V^{\otimes n} \otimes W$ by $\sigma_i \mapsto \mathrm{id}^{\otimes i-1} \otimes \sigma \otimes \mathrm{id}^{n-i}$ for all $1 \le i \le n-1$ and $\tau_n \mapsto \mathrm{id}^{\otimes n-1} \otimes \tau$. Hence, given a left-braided vector space $(V,W,\sigma,\tau)$, $V^{\otimes n} \otimes W$ provides a representation for the group $B_{n,1}$.

For the topological assertions of \cite{h22} and also of this paper, we desire a compatibility condition for the induced representations of mixed braid groups; see Proposition~\ref{prop:comm_diag_cond} and the two paragraphs preceding it for more details. The following definition captures the necessary and sufficient property of a left-braided vector space for the setting in \cite{h22}.

\begin{defn}\label{defn:sep_lbvs}
    A left-braided vector space $(V,W,\sigma,\tau)$ is \textit{separable} if there exists an isomorphism $\varphi : V \otimes W \to W \otimes V$ (called the \textit{separated braiding}) that satisfies the following braid equations on $V^{\otimes 2} \otimes W$:
    \begin{enumerate}
    \item $(\mathrm{id} \otimes \sigma) \circ (\varphi \otimes \mathrm{id}) \circ (\mathrm{id} \otimes \varphi) = (\varphi \otimes \mathrm{id}) \circ (\mathrm{id} \otimes \varphi) \circ (\sigma \otimes \mathrm{id})$;
    \item $(\tau \otimes \mathrm{id}) \circ (\mathrm{id} \otimes \varphi) = (\mathrm{id} \otimes \varphi) \circ (\sigma \otimes \mathrm{id}) \circ (\mathrm{id} \otimes \tau) \circ (\sigma^{-1} \otimes \mathrm{id})$.
    \end{enumerate}
    A separable left-braided vector space $(V,W,\sigma,\tau)$ with the choice of separated braiding $\varphi$ is denoted by $(V,W,\sigma,\tau,\varphi)$.
\end{defn}

We will generalize this construction to give a representation for the $(n,m)$-mixed braid group $B_{n,m}$.

\begin{defn}\label{defn:mix_bvs}
A \textit{mixed-braided vector space} $(V, W, \tau)$ over $\mathbf{k}$ is a pair of braided vector spaces $(V,\sigma_V)$ and $(W,\sigma_W)$, equipped with an isomorphism $\tau : V \otimes W \to V \otimes W$ (called the \textit{mixed braiding}) which satisfies braid equations:
\[(1) \quad (\sigma_V \otimes \mathrm{id}) \circ (\mathrm{id} \otimes \tau) \circ (\sigma_V \otimes \mathrm{id}) \circ (\mathrm{id} \otimes \tau)  = (\mathrm{id} \otimes \tau) \circ (\sigma_V \otimes \mathrm{id}) \circ (\mathrm{id} \otimes \tau) \circ (\sigma_V \otimes \mathrm{id})\]
on $V^{\otimes 2} \otimes W$;
\[(2) \quad (\tau \otimes \mathrm{id}) \circ (\mathrm{id} \otimes \sigma_W) \circ (\tau \otimes \mathrm{id}) \circ (\mathrm{id} \otimes \sigma_W)  = (\mathrm{id} \otimes \sigma_W) \circ (\tau \otimes \mathrm{id}) \circ (\mathrm{id} \otimes \sigma_W) \circ (\tau \otimes \mathrm{id})\]
on $V \otimes W^{\otimes 2}$; and
\[\begin{array}{l}
\hspace{-.06in} (3) \quad (\sigma_V \otimes \mathrm{id}^{\otimes 2}) \circ (\mathrm{id} \otimes \tau \otimes \mathrm{id}) \circ (\sigma_V^{-1} \otimes \mathrm{id}^{\otimes 2}) \circ (\mathrm{id}^{\otimes 2} \otimes \sigma_W) \circ (\mathrm{id} \otimes \tau \otimes \mathrm{id}) \circ (\mathrm{id}^{\otimes 2} \otimes \sigma_W^{-1}) \\[5pt]
= (\mathrm{id}^{\otimes 2} \otimes \sigma_W) \circ (\mathrm{id} \otimes \tau \otimes \mathrm{id}) \circ (\mathrm{id}^{\otimes 2} \otimes \sigma_W^{-1}) \circ (\sigma_V \otimes \mathrm{id}^{\otimes 2}) \circ (\mathrm{id} \otimes \tau \otimes \mathrm{id}) \circ (\sigma_V^{-1} \otimes \mathrm{id}^{\otimes 2})
\end{array}\]
on $V^{\otimes 2} \otimes W^{\otimes 2}$.
\end{defn}

Observe that a mixed-braided vector space $(V,W,\tau)$ is in fact a left-braided vector space where $W$ is braided in such a way that is compatible with the braidings $\sigma_V$ and $\tau$. The mixed-braided vector spaces form a subcategory of the category of left-braided vector spaces, where the morphisms need to satisfy an additional condition that $f_W$ is a morphism of braided vector spaces.

As in the previous constructions, we may define an action of the $(n,m)$-mixed braid group $B_{n,m}$ on $V^{\otimes n} \otimes W^{\otimes m}$ by mapping $\sigma_i$ to $\mathrm{id}^{\otimes i-1} \otimes \sigma_V \otimes \mathrm{id}^{\otimes n+m-1-i}$ if $1 \le i \le n-1$, and $\mathrm{id}^{\otimes i-1} \otimes \sigma_W \otimes \mathrm{id}^{\otimes n+m-1-i}$ if $n+1 \le i \le n+m-1$, and mapping $\tau_n$ to $\mathrm{id}^{\otimes n-1} \otimes \tau \otimes \mathrm{id}^{\otimes m-1}$. From this identification, the following is straightforward:

\begin{prop}\label{prop:mix_rep}
Given a mixed-braided vector space $(V,W,\tau)$, $V^{\otimes n} \otimes W^{\otimes m}$ provides a representation for the $(n,m)$-mixed braid group $B_{n,m}$.
\end{prop}

\begin{defn}\label{defn:sep_mbvs}
A mixed-braided vector space $(V,W,\tau)$ is \textit{left-separable} if it is separable as a left-braided vector space.
\end{defn}

As is the case with \cite{h22}, the technical requirement of separability is essential for our topological arguments in this paper, particularly the proof of Proposition~\ref{prop:F=D}. It is unclear to us whether all mixed-braided vector spaces are left-separable; however, for a few natural examples that are most relevant to potential applications of this paper's topological framework, it is easy to detect a suitable separated braiding.

\begin{exmp}\label{exmp:V=W=k}
Given $V=W=\mathbf{k}$, we can define a mixed-braided vector space $(V,W,\tau)$ with braidings $\sigma_V, \sigma_W$, and $\tau$ given by multiplications by $q$, $u$, and $p$ respectively, for some $p, q, u \in \mathbf{k}^\times$. 
This mixed-braided vector space is left-separable when the separated braiding $\varphi: V \otimes W \to W \otimes V$ is chosen to be (any invertible multiple of) permutation of tensor factors. The braid action of $B_{n,m}$ on the representation $V^{\otimes n} \otimes W^{\otimes m} \cong \mathbf{k}$ is therefore given by $\sigma_i \mapsto q$ for $1 \le i \le n-1$, $\sigma_i \mapsto u$ for $n+1 \le i \le n+m-1$, and $\tau_n \mapsto p$. 
\end{exmp}

\begin{exmp}
If $(V,\sigma)$ is a braided vector space, then $(V,V,\sigma^2)$ forms a left-separable mixed-braided vector space with the obvious choice of separated braiding $\varphi := \sigma$. The $B_{n,m}$-representation constructed from this mixed-braided vector space per Proposition~\ref{prop:mix_rep} is precisely the restricted representation to $B_{n,m}$ of the monoidal braid representation of $B_{n+m}$ on $V^{\otimes n+m}$.
\end{exmp}


\subsection{Induced representations of braid subgroups}

Recall that an $(n,m)$-shuffle is a permutation on $n+m$ elements that preserves the order on the first $n$ and the last $m$ elements. An $(n,m)$-shuffle $\gamma$ can be completely determined by an indexing set $\mathcal{I} = \{i_1, \dots, i_m\}$ where $i_k = \gamma(n+k)$ $(1 \le i_1 < \dots < i_m \le n+m)$. We denote the $(n,m)$-shuffle defined by the indexing set $\mathcal{I}$ by $\gamma_{\mathcal{I},n+m}$; the second index will usually be omitted when there is no ambiguity. If $\mathcal{I} = (i_1, \dots, i_m)$ is an $m$-tuple whose elements are distinct but not necessarily increasing, we define the associated shuffle $\gamma_{\mathcal{I},n+m}$ by the underlying set of $\mathcal{I}$.

Consider the left cosets of $B_{n,m}$ in $B_{n+m}$.

\begin{prop}
The collection of left cosets of $B_{n,m}$ in $B_{n+m}$ has the form
\[B_{n+m}/B_{n,m} = \{\widetilde{\alpha} B_{n,m} : \alpha \in \mathrm{Sh}(n,m)\}.\]
\end{prop}

\begin{proof}
We claim that the left cosets $a B_{n,m}$ are indexed by the image of the integer interval $\llbracket n+1, n+m \rrbracket$ under the underlying permutation of the braid $a \in B_{n+m}$. For any $a_1, a_2 \in B_{n+m}$, $a_1 B_{n,m} = a_2 B_{n,m}$ as cosets iff $a_1^{-1} a_2 \in B_{n,m}$. By definition, this is equivalent to $\underline{a_1}^{-1} \underline{a_2} (\llbracket n+1, n+m \rrbracket) = \llbracket n+1, n+m \rrbracket$, or $\underline{a_1}(\llbracket n+1, n+m \rrbracket) = \underline{a_2}(\llbracket n+1, n+m \rrbracket)$. So we have a simple characterization of the cosets of $B_{n,m}$ in $B_{n+m}$: two braid elements of $B_{n+m}$ are in the same coset of $B_{n,m}$ if and only if the images of the interval $\llbracket n+1, n+m \rrbracket$ under their underlying permutations coincide. It follows that the index of $B_{n,m}$ in $B_{n+m}$ is $n+m \choose m$. Furthermore, if we impose that the underlying permutations of the representative braids preserve the order on $\llbracket n+1, n+m \rrbracket$, observe that an explicit choice for the representatives of the cosets of $B_{n,m}$ is the collection of the lifts of all $(n,m)$-shuffles $\alpha$, as desired.
\end{proof}

Since $\{\widetilde{\alpha} : \alpha \in \mathrm{Sh}(n,m) \}$ forms a full set of representatives, for each $a \in B_{n+m}$ and each $\widetilde{\alpha}$, there exist uniquely elements $b \in B_{n,m}$ and $\widetilde{\alpha'}$ such that $a \widetilde{\alpha} = \widetilde{\alpha'} b$. We may give a concrete description of these elements. For any $(n,m)$-shuffle $\alpha_\mathcal{I}$ and $p \in S_{n+m}$, let $p_{\alpha_\mathcal{I}}$ denote the $(n,m)$-shuffle associated to $p(\mathcal{I}) = (p(i_1), \dots, p(i_m))$. Note that in general, $p_{\alpha_\mathcal{I}} = p\alpha_\mathcal{I}$ if and only if $p$ preserves the order on $\mathcal{I}$ and $[n+m]\setminus\mathcal{I}$. Since the underlying permutation of $b = (\widetilde{\alpha'})^{-1} a \widetilde{\alpha_\mathcal{I}} \in B_{n,m}$ preserves $\llbracket n+1, n+m \rrbracket$, it follows that $\alpha' (\llbracket n+1,n+m \rrbracket) = \underline{a}\alpha_\mathcal{I} (\llbracket n+1,n+m \rrbracket) = \underline{a}(\mathcal{I})$ as unordered sets. By the above definition, $\alpha' = \underline{a}_{\alpha_\mathcal{I}}$, so $b = \widetilde{{\underline{a}_\alpha}}^{-1} a \widetilde{\alpha}$.

In the topological setup of this paper (see Sections~\ref{subsec:fnf_Cm} and~\ref{subsec:fnf_twisted}), braids often arise from lifting shuffles on $n+m$ elements that preserve the order on the set $\mathcal{I}$ of overall positions of the fixed points in a configuration. As a result, we are particularly interested in the case when the $i_j^\mathrm{th}$ and $i_k^\mathrm{th}$ strands of the braid $a$ are pairwise parallel for all distinct $i_j, i_k \in \mathcal{I}$. In this case, for all $n+1 \le j < k \le n+m$, the $j^\mathrm{th}$ and $k^\mathrm{th}$ strands of $b$ are always parallel in each successive component $\widetilde{\alpha}$, $a$, and $\widetilde{\underline{a}_\alpha}^{-1}$. This property allows us to straighten each of the last $m$ strands in the braid, which implies that $b \in B_n(\mathbb{C}_m)$. Moreover, since $\underline{a}$ preserves the order on $\mathcal{I}$, the tuple $\underline{a}(\mathcal{I}) = \left(\underline{a}(i_1), \dots, \underline{a}(i_m) \right)$ already has the desired increasing order defining the shuffle $\underline{a}_\alpha$.

Given a representation of any subgroup, we may define a representation of the parent group by means of the \textit{induced representation}. Let $L$ be a representation of $B_{n,m}$. The braid representation of $B_{n,m}$ on $L$ induces a representation on
\[\mathrm{Ind}^{B_{n+m}}_{B_{n,m}} (L) = k[B_{n+m}] \displaystyle \otimes_{k[B_{n,m}]} L \]
of the braid group $B_{n+m}$. We may give a more detailed description of this induced representation based on the cosets of the subgroup $B_{n,m}$ in $B_{n+m}$ described above. Since the collection $\{\widetilde{\alpha} : \alpha \in \mathrm{Sh}(n,m) \}$ gives a full set of representatives in $B_{n+m}$ for the left cosets of $B_{n,m}$, as vector spaces, the induced representation can be identified as
\[\mathrm{Ind}^{B_{n+m}}_{B_{n,m}} (L) \cong \displaystyle \bigoplus_{\alpha \in \mathrm{Sh}(n,m)} \widetilde{\alpha} L.\]
Here each $\widetilde{\alpha} L$ is an isomorphic copy of the vector space $L$ whose elements are written as $\widetilde{\alpha} \ell$ where $\ell \in L$. An explicit formula for the action of the braid group on this induced representation follows immediately from the previous paragraphs.

\begin{prop}\label{prop:act_ind_mix}
The action of the braid group $B_{n+m}$ on the induced representation $\mathrm{Ind}^{B_{n+m}}_{B_{n,m}} (L)$ is given by
\[ a \sum_{\alpha\in\mathrm{Sh}(n,m)} \widetilde{\alpha} \ell_\alpha = \sum_{\alpha\in\mathrm{Sh}(n,m)} \widetilde{\underline{a}_\alpha} \big[ (\widetilde{\underline{a}_\alpha}^{-1} a \widetilde{\alpha}) (\ell_\alpha) \big]. \]
\end{prop}

The following computation is straightforward:

\begin{cor}\label{cor:act_ind_mix_gen}
Let $\mathcal{I} = (i_1, \dots, i_m)$ be the indexing set of an $(n,m)$-shuffle $\alpha$, and $\mathcal{I}_{k,h} = (i_1, \dots, i_k + h, \dots, i_m)$. The action of the generators of $B_{n+m}$ on $\mathrm{Ind}^{B_{n+m}}_{B_{n,m}}(L)$ can be expressed in terms of the action of the generators of $B_{n,m}$ in the following way:
\[ \sigma_i(\widetilde{\alpha} \ell) = \begin{cases}
\widetilde{\alpha} [\sigma_i (\ell)] \hfill 1 \le i < i_1-1 \\
\widetilde{\alpha} [\sigma_{i-k} (\ell)] \hspace{1.6in} \hfill 1 \le k \le m-1, i_k < i < i_{k+1}-1 \\
\widetilde{\alpha} [\sigma_{i-m} (\ell)] \hfill i_m < i < n+m \\
\widetilde{\alpha}[\sigma_{n+k}(\ell)] \hfill 1 \le k \le m-1, i = i_k = i_{k+1}-1 \\
\widetilde{\alpha_{\mathcal{I}_{k,-1}}} \ell \hfill 1 \le k \le m, i = i_k-1 > i_{k-1} \\
\widetilde{\alpha_{\mathcal{I}_{k,1}}} [(\sigma_{i_k-k+1} \dots \sigma_{n-1} \theta_{n,n+k} \sigma_{n-1}^{-1} \dots \sigma_{i_k-k+1}^{-1})(\ell)] \\
\hfill 1 \le k \le m, i = i_k < i_{k+1}-1,
\end{cases}
\]
where $\theta_{n,n+k} = \sigma_{n+k-1} \dots \sigma_{n+1} \tau_n \sigma_{n+1}^{-1} \dots \sigma_{n+k-1}^{-1}$.
\end{cor}

Let $(V,W,\tau)$ be a mixed-braided vector space. Recall that $V^{\otimes n} \otimes W^{\otimes m}$ gives a representation of $B_{n,m}$. Consider the induced representation $\mathrm{Ind}_{B_{n,m}}^{B_{n+m}}(V^{\otimes n} \otimes W^{\otimes m})$.
Each summand $\widetilde{\alpha} (V^{\otimes n}\otimes W^{\otimes m})$ of this induced representation is isomorphic to $V^{\otimes n}\otimes W^{\otimes m}$. Since there is a one-to-one correspondence between the set of $(n,m)$-shuffles and tuples $\mathcal{I} = (i_1, \dots, i_m)$ of $m$ strictly increasing integers in $\llbracket 1, n+m \rrbracket$, it is natural to identify $\widetilde{\alpha_{\mathcal{I},n+m}} (V^{\otimes n}\otimes W^{\otimes m})$ with $V^{\otimes i_1-1} \otimes W \otimes V^{\otimes i_2-i_1-1} \otimes W \otimes \dots \otimes W \otimes V^{\otimes n+m-i_m}$, where the $i_k^\mathrm{th}$ tensor factor is $W$ for all $1 \le k \le m$, via an isomorphism $\xi_{\mathcal{I},n+m} : \widetilde{\alpha_{\mathcal{I},n+m}} (V^{\otimes n}\otimes W^{\otimes m}) \xrightarrow{\cong} V^{\otimes i_1-1} \otimes W \otimes V^{\otimes i_2-i_1-1} \otimes W \otimes \dots \otimes W \otimes V^{\otimes n+m-i_m}$.

\begin{prop}\label{prop:ind_mix}
There is an isomorphism of vector spaces
\[ \mathrm{Ind}_{B_{n,m}}^{B_{n+m}}(V^{\otimes n}\otimes W^{\otimes m}) \cong \displaystyle \bigoplus_{\mathcal{I}} V^{\otimes i_1-1} \otimes W \otimes V^{\otimes i_2-i_1-1} \otimes W \otimes \dots \otimes W \otimes V^{\otimes n+m-i_m}\]
where $\mathcal{I}$ runs over all tuples $(i_1, \dots, i_m)$ with $1 \le i_1 < \dots < i_m \le n+m$.

Moreover, given a choice of isomorphisms $\xi_{\mathcal{I},n+m}$, there is a $B_{n+m}$-action on the right hand side defined by $a \mapsto \xi_{\underline{a}(\mathcal{I})} a \xi_\mathcal{I}^{-1}$, such that the above is an isomorphism of $B_{n+m}$-representation.
\end{prop}

\begin{proof}
The first statement results directly from the previous paragraph. The second follows immediately from the definition of the action.
\end{proof}

By convention, it is always assumed that $\xi_{\mathcal{I},n+m}$ is the identity if $\mathcal{I} = (n+1,\dots,n+m)$. The second index of the map $\xi_{\mathcal{I},n+m}$ again denotes the total degree of the domain and is omitted from the notation if there is no ambiguity. Observe that on the right hand side, we can apply braids that ``swap'' $V$ and $W$, an operation that is forbidden in the $B_{n,m}$-representation on $V^{\otimes n} \otimes W^{\otimes m}$. This allows for a more intuitive framework to study the action of $B_{n+m}$ on the induced representation, more analogous to its action in the monoidal braid representation on $V^{\otimes n+m}$.
As in the quantum shuffle algebra, we will also denote elements of $V^{\otimes i_1-1} \otimes W \otimes V^{\otimes i_2-i_1-1} \otimes W \otimes \dots \otimes W \otimes V^{\otimes n+m-i_m}$ by the bar complex notation, i.e., $[v_1|\dots|v_{i_k-1}| w_k| v_{i_k+1}| \dots| v_{n+m}]$.

Recall that for the purpose of this paper, we are interested the action of braids $a$ whose $i_j^\mathrm{th}$ and $i_k^\mathrm{th}$ strands are pairwise parallel for all distinct $i_j, i_k \in \mathcal{I}$. In particular, the order on $\mathcal{I}$ is preserved throughout $a$. It follows that generators in the standard decomposition of $a$, where each crossing in $a$ corresponds to a braid generator or its inverse, never swap two copies of $W$ on the right hand side of Proposition~\ref{prop:ind_mix}. It is therefore suggestive to denote the $k^\mathrm{th}$ occurrence of $W$ by $W_k$; we will adopt this convention whenever this case applies.

Let $(V,W,\tau,\varphi)$ be a left-separable mixed-braided vector space. Recall from \cite{h22} that separability of $(V,W)$ as a left-braided vector space is the necessary and sufficient condition for the $B_{n+1}$-representation on $\bigoplus_{i=1}^{n+1} V^{\otimes i-1} \otimes W \otimes V^{\otimes n-i+1}$ to behave analogously to that on $V^{\otimes n+1}$. That is, if we define
\[\varphi_{i,n} := (\mathrm{id}^{\otimes i-1} \otimes \varphi \otimes \mathrm{id}^{\otimes n-i-1}) \circ (\mathrm{id}^{\otimes i} \otimes \varphi \otimes \mathrm{id}^{\otimes n-i-2}) \circ \cdots \circ (\mathrm{id}^{\otimes n-2} \otimes \varphi)\]
for all $1 \le i \le n$, then the following diagram commutes for all $1 \le q \le n$, $a \in B_{q+1}$, and $1 \le i \le q+1$ with the choice of isomorphisms $\xi_{i,n} = \varphi_{i,n} \widetilde{\alpha_{i,n}}^{-1}$:
\[ \begin{tikzcd}
V^{\otimes p} \otimes (V^{\otimes i-1} \otimes W \otimes V^{q-i+1}) \otimes V^{\otimes n-p-q} \arrow[swap]{d}{\mathrm{id}^{\otimes p} \otimes \xi^{-1}_{i,q+1} \otimes \mathrm{id}^{\otimes n-p-q}} \arrow{r}{\xi_{p+i,n+1}^{-1}}  &[0.7cm] \widetilde{\alpha_{p+i,n+1}}(V^{\otimes n} \otimes W) \arrow{ddd}{a'} \\%
V^{\otimes p} \otimes \widetilde{\alpha_{i,q+1}} (V^{\otimes q} \otimes W) \otimes V^{\otimes n-p-q} \arrow[swap]{d}{\mathrm{id}^{\otimes p} \otimes a \otimes \mathrm{id}^{\otimes n-p-q}} \\
V^{\otimes p} \otimes \widetilde{\alpha_{\underline{a}(i),q+1}} (V^{\otimes q} \otimes W) \otimes V^{\otimes n-p-q} \arrow[swap]{d}{\mathrm{id}^{\otimes p} \otimes \xi_{\underline{a}(i),q+1} \otimes \mathrm{id}^{\otimes n-p-q}}\\
V^{\otimes p} \otimes (V^{\otimes \underline{a}(i)-1} \otimes W \otimes V^{q-\underline{a}(i)+1}) \otimes V^{\otimes n-p-q} & \arrow[swap]{l}{\xi_{\underline{a'}(p+i),n+1}} \widetilde{\alpha_{\underline{a'}(p+i),n+1}}(V^{\otimes n} \otimes W)
\end{tikzcd} \]
where the braid $a'$ is the natural inclusion of $a$ into the copy $B_{q+1} \le B_{n+1}$ consisting of braids that are only nontrivial on the $q+1$ strands starting with the $p+1^{\mathrm{st}}$.

A similar property is desired for the $B_{n+m}$-action on $\bigoplus_{\mathcal{I}} V^{\otimes i_1-1} \otimes W \otimes V^{\otimes i_2-i_1-1} \otimes W \otimes \dots \otimes W \otimes V^{\otimes n+m-i_m}$. That is, for a fixed $m \ge 1$ and any $n \ge 1$, $1 \le j \le q \le n+m$, and $a \in B_q$, the diagram

\begin{equation}\label{eq:comm_diag}
\begin{tikzcd}
Y^{\otimes p} \otimes (V^{\otimes j-1} \otimes W_k \otimes V^{q-j}) \otimes Y^{\otimes n+m-p-q} \arrow[swap]{d}{\mathrm{id}^{\otimes p} \otimes \xi^{-1}_{j,q} \otimes \mathrm{id}^{\otimes n+m-p-q}} \arrow{r}{\xi_{\mathcal{I},n+m}^{-1}} &[0.7cm] \widetilde{\alpha_{\mathcal{I},n+m}}(V^{\otimes n} \otimes W^{\otimes m}) \arrow{ddd}{a'}\\%
Y^{\otimes p} \otimes \widetilde{\alpha_{j,q}} (V^{\otimes q-1} \otimes W) \otimes Y^{\otimes n+m-p-q} \arrow[swap]{d}{\mathrm{id}^{\otimes p} \otimes a \otimes \mathrm{id}^{\otimes n+m-p-q}} \\
Y^{\otimes p} \otimes \widetilde{\alpha_{\underline{a}(j),q}} (V^{\otimes q-1} \otimes W) \otimes Y^{\otimes n+m-p-q} \arrow[swap]{d}{\mathrm{id}^{\otimes p} \otimes \xi_{\underline{a}(j),q} \otimes \mathrm{id}^{\otimes n+m-p-q}}\\
Y^{\otimes p} \otimes (V^{\otimes \underline{a}(j)-1} \otimes W_k \otimes V^{q-\underline{a}(j)}) \otimes Y^{\otimes n+m-p-q} & \arrow[swap]{l}{\xi_{\underline{a'}(\mathcal{I}),n+m}} \widetilde{\alpha_{\underline{a'}(\mathcal{I}),n+m}}(V^{\otimes n} \otimes W^{\otimes m})
\end{tikzcd}
\end{equation}

\noindent commutes where $p = i_k-j$, $Y$ denotes a copy of either $V$ or $W$, and the braid $a'$ is the natural inclusion of $a$ into the copy $B_q \le B_{n+m}$ consisting of braids that are only nontrivial on the $q$ strands starting with the $p+1^{\mathrm{st}}$. Roughly speaking, we want the inclusion of the subspace $\bigoplus_{j=1}^{q} V^{\otimes j-1} \otimes W_k \otimes V^{\otimes q-j}$ into $\bigoplus_{\mathcal{I}} V^{\otimes i_1-1} \otimes W_1 \otimes V^{\otimes i_2-i_1-1} \otimes W_2 \otimes \dots \otimes W_m \otimes V^{\otimes n+m-i_m}$ to be equivariant with respect to the braid action. The following proposition gives criteria to detect this property.

\begin{prop}\label{prop:comm_diag_cond}
Let $(V,W,\tau,\varphi)$ be a left-separable mixed-braided vector space. For $\mathcal{I} = (i_1,\dots,i_m)$, let $\varphi_{\mathcal{I},n+m} : V^{\otimes n} \otimes W^{\otimes m} \to V^{\otimes i_1-1} \otimes W_1 \otimes V^{\otimes i_2-i_1-1} \otimes W_2 \otimes \dots \otimes W_m \otimes V^{\otimes n+m-i_m}$ be defined by $\varphi_{\mathcal{I},n+m}:= \xi_{\mathcal{I},n+m}\widetilde{\alpha_{\mathcal{I},n+m}}$ for all $n \ge 1$. Then Diagram~\ref{eq:comm_diag} always commutes if and only if the following identities hold:
\begin{enumerate}
    \item $\varphi_{\mathcal{I},n+m} = \varphi_{i_m, n+m} \circ \cdots \circ (\varphi_{i_2,n+2} \otimes \mathrm{id}^{\otimes m-2}) \circ (\varphi_{i_1,n+1} \otimes \mathrm{id}^{\otimes m-1})$;
    \item $(\mathrm{id} \otimes \tau) \circ (\varphi \otimes \mathrm{id}) = (\varphi \otimes \mathrm{id}) \circ (\mathrm{id} \otimes \sigma_W) \circ (\tau \otimes \mathrm{id}) \circ (\mathrm{id} \otimes \sigma_W^{-1})$.
\end{enumerate}
\end{prop}

\begin{proof}
Since every braid action is decomposable into those of the generators, it suffices to study the commutativity of Diagram~\ref{eq:comm_diag} for all braid generators. By brute force, we observe that Diagram~\ref{eq:comm_diag} commutes for all braid generators if and only if the following identities hold:
\begin{enumerate}[label=(\alph*)]
    \item $(\mathrm{id}^{\otimes p} \otimes \varphi_{j-p,i_k-p} \otimes \mathrm{id}^{\otimes n+m-i_k}) \varphi_{\mathcal{I},n+m} = \varphi_{\mathcal{I}_{k,j-i_k},n+m}$, for $0 \le p < j$ and $i_{k-1} < j \le i_k$;
    
    \vspace{.05in}
    \item $\sigma_i \varphi_{\mathcal{I},n+m} = \begin{cases}
    \varphi_{\mathcal{I},n+m} \sigma_i \qquad \text{if } 1 \le i \le i_1-2\\
    \varphi_{\mathcal{I},n+m} \sigma_{i-k} \quad \text{if } i_k +1 \le i \le i_k-2\\
    \varphi_{\mathcal{I},n+m} \sigma_{i-m} \quad \text{if } i_m+1 \le i < n+m
    \end{cases}$
    \newline for $\sigma_i = \mathrm{id}^{\otimes i-1} \otimes \sigma \otimes \mathrm{id}^{\otimes n+m-i-1}$;
    
    \item $\tau_{i_k-1} \varphi_{\mathcal{I},n+m} = \varphi_{\mathcal{I},n+m} (\sigma_{i_k-k} \dots \sigma_{n-1} \theta_{n,n+k} \sigma_{n-1}^{-1} \dots \sigma_{i_k-k}^{-1})$, for $\tau_{i} = \mathrm{id}^{\otimes i-1} \otimes \tau \otimes \mathrm{id}^{\otimes n+m-i-1}$.
\end{enumerate}
We will show that these conditions are equivalent to equations (1-2).

First, we will prove that condition (a) is equivalent to formula (1) given the formula
\[\varphi_{i,n} = (\mathrm{id}^{\otimes i-1} \otimes \varphi \otimes \mathrm{id}^{\otimes n-i-1}) \circ (\mathrm{id}^{\otimes i} \otimes \varphi \otimes \mathrm{id}^{\otimes n-i-2}) \circ \cdots \circ (\mathrm{id}^{\otimes n-2} \otimes \varphi).\]
It is easy to verify that formula (1) for $\varphi_{\mathcal{I},n+m}$ satisfies condition (a). Conversely, we will prove formula (1) by induction on the smallest $k \ge 1$ such that $i_k = n+k$ (which implies that $i_r = n+r$ for all $k \le r \le m$). The base case $k = 1$ is trivial, as both sides are the identity. Suppose for all $\mathcal{I} = (i_1, \dots, i_{k-1}, n+k, \dots, n+m)$,
\[ \begin{array}{r l}
\varphi_{\mathcal{I},n+m} & = \varphi_{i_m, n+m} \circ \dots \circ (\varphi_{i_k,n+k} \otimes \mathrm{id}^{\otimes m-k}) \circ (\varphi_{i_{k-1},n+k-1} \otimes \mathrm{id}^{\otimes m-k+1}) \circ \dots \\[5pt]
& \quad \null \circ (\varphi_{i_1,n+1} \otimes \mathrm{id}^{\otimes m-1})\\[5pt]
& = (\varphi_{i_{k-1},n+k-1} \otimes \mathrm{id}^{\otimes m-k+1}) \circ \dots \circ (\varphi_{i_1,n+1} \otimes \mathrm{id}^{\otimes m-1}).
\end{array}\]
Observe that any $\mathcal{I'} = (i_1, \dots, i_{k-1}, i_k, n+k+1, \dots, n+m)$ can be written as $\mathcal{I}_{k,i_k-(n+k)}$, then by identity (a) we have
\[\begin{array}{r l}
     \varphi_{\mathcal{I'},n+m} & = \varphi_{\mathcal{I}_{k,i_k-(n+k)},n+m} = (\varphi_{i_k,n+k} \otimes \mathrm{id}^{\otimes n+m-(n+k)}) \varphi_{\mathcal{I},n+m}\\[5pt]
     & = (\varphi_{i_k,n+k} \otimes \mathrm{id}^{\otimes m-k}) \circ (\varphi_{i_{k-1},n+k-1} \otimes \mathrm{id}^{\otimes m-k+1}) \circ \dots \circ (\varphi_{i_1,n+1} \otimes \mathrm{id}^{\otimes m-1})
\end{array}\]
which proves the claim for case $k+1$. With this identification, observe that condition (c) gives $\tau_2 \varphi_{(1,3,\dots,m+1),m+1} = \varphi_{(1,3,\dots,m+1),m+1} \theta_{1,3}$, which is identity (2). We have thus proved the forward direction.

The converse can be proved using a similar induction argument. Condition (b) is implied by formula (1) and the first three cases of Corollary~\ref{cor:act_ind_mix_gen}. Meanwhile, condition (c) follows from separability of $(V,W,\tau,\varphi)$, equation (2), and the last case of the same corollary.
\end{proof}

Left-separability of a mixed-braided vector space $(V,W)$ and Equation~\ref{prop:comm_diag_cond}.2 are integrally connected to the existence of a braid structure on the direct sum $V \oplus W$.

\begin{prop}
Let $V$ and $W$ be finite dimensional $\mathbf{k}$-vector spaces, and let $X = V \oplus W$. Suppose there is an automorphism $\sigma_X$ of $(V \oplus W)^{\otimes 2} \cong V^{\otimes 2} \oplus (V \otimes W) \oplus (W \otimes V) \oplus W^{\otimes 2}$ defined summand-wise by isomorphisms $\sigma_V: V^{\otimes 2} \to V^{\otimes 2}$, $\varphi: V \otimes W \to W \otimes V$, $\psi: W \otimes V \to V \otimes W$, and $\sigma_W: W^{\otimes 2} \to W^{\otimes 2}$.
\begin{enumerate}
    \item If $(X,\sigma_X)$ is a braided vector space, then $((V,\sigma_V),(W,\sigma_W),\tau,\varphi)$ is a left-separable mixed-braided vector space, where $\tau = \psi\varphi$, that satisfies Equation~\ref{prop:comm_diag_cond}.2;
    
    \item A weak version of the converse holds: if $(V,W,\tau,\varphi)$ is a left-separable mixed-braided vector space that satisfies Equation~\ref{prop:comm_diag_cond}.2 and
    \[(\sigma_W \otimes \mathrm{id}) \circ (\mathrm{id} \otimes \varphi) \circ (\varphi \otimes \mathrm{id}) = (\mathrm{id} \otimes \varphi) \circ (\varphi \otimes \mathrm{id}) \circ (\mathrm{id} \otimes \sigma_W)\]
    then $(X,\sigma_X)$ is a braided vector space.
\end{enumerate}
\end{prop}

\begin{proof}
This proof rests on the following key observation: $(X,\sigma_X)$ is a braided vector space if and only if the braid equation holds on each of the eight direct summands of $(V \oplus W)^{\otimes 3}$. That is,
\begin{enumerate}[label=(\alph*)]
    \item $(\sigma_V \otimes \mathrm{id}) \circ (\mathrm{id} \otimes \sigma_V) \circ (\sigma_V \otimes \mathrm{id}) = (\mathrm{id} \otimes \sigma_V) \circ (\sigma_V \otimes \mathrm{id}) \circ (\mathrm{id} \otimes \sigma_V)$; 
    \item $(\varphi \otimes \mathrm{id}) \circ (\mathrm{id} \otimes \varphi) \circ (\sigma_V \otimes \mathrm{id}) = (\mathrm{id} \otimes \sigma_V) \circ (\varphi \otimes \mathrm{id}) \circ (\mathrm{id} \otimes \varphi)$; 
    \item $(\psi \otimes \mathrm{id}) \circ (\mathrm{id} \otimes \sigma_V) \circ (\varphi \otimes \mathrm{id}) = (\mathrm{id} \otimes \varphi) \circ (\sigma_V \otimes \mathrm{id}) \circ (\mathrm{id} \otimes \psi)$; 
    \item $(\sigma_V \otimes \mathrm{id}) \circ (\mathrm{id} \otimes \psi) \circ (\psi \otimes \mathrm{id}) = (\mathrm{id} \otimes \psi) \circ (\psi \otimes \mathrm{id}) \circ (\mathrm{id} \otimes \sigma_V)$; 
    \item $(\sigma_W \otimes \mathrm{id}) \circ (\mathrm{id} \otimes \varphi) \circ (\varphi \otimes \mathrm{id}) = (\mathrm{id} \otimes \varphi) \circ (\varphi \otimes \mathrm{id}) \circ (\mathrm{id} \otimes \sigma_W)$;
    \item $(\varphi \otimes \mathrm{id}) \circ (\mathrm{id} \otimes \sigma_W) \circ (\psi \otimes \mathrm{id}) = (\mathrm{id} \otimes \psi) \circ (\sigma_W \otimes \mathrm{id}) \circ (\mathrm{id} \otimes \varphi)$;
    \item $(\psi \otimes \mathrm{id}) \circ (\mathrm{id} \otimes \psi) \circ (\sigma_W \otimes \mathrm{id}) = (\mathrm{id} \otimes \sigma_W) \circ (\psi \otimes \mathrm{id}) \circ (\mathrm{id} \otimes \psi)$;
    \item $(\sigma_W \otimes \mathrm{id}) \circ (\mathrm{id} \otimes \sigma_W) \circ (\sigma_W \otimes \mathrm{id}) = (\mathrm{id} \otimes \sigma_W) \circ (\sigma_W \otimes \mathrm{id}) \circ (\mathrm{id} \otimes \sigma_W)$. 
\end{enumerate}
By Proposition 2.17 of \cite{h22}, conditions (a-d) are equivalent to the fact that $(V,W,\sigma_V,\tau,\varphi)$ forms a separable left-braided vector space, and condition (h) is equivalent to $(W,\sigma_W)$ being a braided vector space. The added relation in part (2) of the statement is precisely condition (e), so it suffices to show that conditions (f-g) are equivalent to 
Equation~\ref{prop:comm_diag_cond}.2 and the extra braid relations~\ref{defn:mix_bvs}.2 and ~\ref{defn:mix_bvs}.3, under the assumptions of (a-e, h).
The arguments for both directions are very similar; here we will only show the proof of part (1) of the statement. For the rest of this proof, we will use the notation $\sigma^V_i := \mathrm{id}^{\otimes i-1} \otimes \sigma_V \otimes \mathrm{id}^{\otimes *}$, and analogous notations for $\sigma_W$, $\tau$, $\varphi$, and $\psi$.

Assume identities (a-h) hold. First, we will show that conditions (e-f) imply Equation~\ref{prop:comm_diag_cond}.2. Apply (e) and (f) subsequently to the following:
\[\varphi_1^{-1} \tau_2 \varphi_1 \sigma^W_2 = \varphi_1^{-1} \psi_2 (\varphi_2 \varphi_1 \sigma^W_2) = \varphi_1^{-1} (\psi_2 \sigma^W_1 \varphi_2) \varphi_1 = (\varphi_1^{-1} \varphi_1) \sigma^W_2 \psi_1 \varphi_1 = \sigma^W_2 \tau_1.\]
It follows that $\tau_2 \varphi_1 = \varphi_1 \sigma^W_2 \tau_1 (\sigma^W_2)^{-1}$, so Equation~\ref{prop:comm_diag_cond}.2 holds. We then show conditions (e-g) imply Equation~\ref{defn:mix_bvs}.2 as follows:
\[\begin{array}{r l}
     \tau_1 \sigma^W_2 \tau_1 \sigma^W_2 & = \psi_1 (\varphi_1 \sigma_2^W \psi_1) \varphi_1 \sigma_2^W = (\psi_1 \psi_2 \sigma_1^W) (\varphi_2 \varphi_1 \sigma_2^W)\\[5pt]
     & = \sigma^W_2 \psi_1 (\psi_2 \sigma_1^W \varphi_2) \varphi_1 = \sigma^W_2 (\psi_1 \varphi_1) \sigma_2^W (\psi_1 \varphi_1) = \sigma^W_2 \tau_1 \sigma^W_2 \tau_1.
\end{array}\]
Finally, we show that Equation~\ref{prop:comm_diag_cond}.2 and Equation~\ref{defn:sep_lbvs}.2 (separability of the left-braided vector space $(V,W,\sigma_V,\tau,\varphi)$) imply Equation~\ref{defn:mix_bvs}.3:
\[\begin{array}{l}
\left( \sigma^V_1 \tau_2 (\sigma^V_1)^{-1} \right) \left(\sigma^W_3 \tau_2 (\sigma^W_3)^{-1} \right) = \left( \varphi_2^{-1} \tau_1 \varphi_2 \right) \left( \varphi_2^{-1} \tau_3 \varphi_2 \right) = \varphi_2^{-1} \tau_1 \tau_3 \varphi_2 \\[5pt]
= \varphi_2^{-1} \tau_3 \tau_1 \varphi_2 = \left( \varphi_2^{-1} \tau_3 \varphi_2 \right) \left( \varphi_2^{-1} \tau_1 \varphi_2 \right) = \left( \sigma^W_3 \tau_2 (\sigma^W_3)^{-1} \right) \left( \sigma^V_1 \tau_2 (\sigma^V_1)^{-1} \right).
\end{array}\]
This completes the proof of part (1).
\end{proof}


\section{Cellular homology of configuration spaces}\label{sec:H*_conf}

In this section, we will develop a framework to compute the cellular homology of $\mathrm{Conf}_n(\mathbb{C}_m)$ with coefficients in local systems. As our main topological result, we will identify the (co)homology of $\mathrm{Conf}_n(\mathbb{C}_m)$ with coefficients in the local system associated to the $B_n(\mathbb{C}_m)$-representation on $V^{\otimes n}\otimes W^{\otimes m}$, with the (co)homology of certain bimodules $\mathfrak{M}$ defined over the quantum shuffle algebra $\mathfrak{A}(V_\epsilon)$ (Definition~\ref{defn:qsa}). As an application, we will compute the homology of $\mathrm{Conf}_n(\mathbb{C}_m)$ and in turn prove a vanishing range for the homology of $\mathrm{Conf}_{n,m}(\mathbb{C})$ with coefficients in rank-$1$ local systems associated to characters of the resultant.


\subsection{Cellular chain complex with local coefficients}\label{subsec:fnf_twisted}

Let $\mathcal{I} = (i_1, \dots, i_m)$ where $1 \le i_1 < \dots < i_m \le n+m$, and $\alpha_\mathcal{I}$ the associated $(n,m)$-shuffle. Define a function of sets $\eta_\mathcal{I}: B_{n+m} \to B_{n,m}$ by sending $a$ to $\widetilde{\underline{a}_{\alpha_\mathcal{I}}}^{-1} a \widetilde{\alpha_\mathcal{I}}$, i.e., the unique element $b$ such that $a \widetilde{\alpha_\mathcal{I}} = \widetilde{\alpha'} b$ for some $\alpha' \in \mathrm{Sh}(n,m)$. This map is not a group homomorphism; however, it satisfies the composition relation $\eta_\mathcal{I} (ba) = \eta_{\underline{a}(\mathcal{I})}(b) \eta_\mathcal{I} (a)$. When restricting the domain to only braids with pairwise parallel $i_j^\mathrm{th}$ and $i_k^\mathrm{th}$ strands for all distinct $i_j, i_k \in \mathcal{I}$, we observed that the range is in fact $B_n(\mathbb{C}_m)$. Since our braids always arise from lifting shuffles that preserve the order on the set $\mathcal{I}$ of overall positions of the fixed points in a configuration, this assumption applies to our discussion below.

Let $L$ be a representation of $B_n(\mathbb{C}_m)$, and $\mathcal{L}$ be the associated local system over Conf$_{n}(\mathbb{C}_m)$. Since $\mathcal{L}$ trivializes on the open cells of the Fox-Neuwirth stratification for Conf$_n (\mathbb{C}_m)$, the cellular chain complex with local coefficients $C_*(\mathrm{Conf}_*(\mathbb{C}_m) \cup \{\infty\}, \{\infty\}; \mathcal{L})$ is isomorphic to $D(n,m)_* \otimes L$ as graded groups. The differential of $D(n,m)_* \otimes L$ which incorporates the braid action on $L$ is defined by
\[\begin{array}{r l}
    \displaystyle d[(\lambda, I, J) \otimes \ell)] & = \displaystyle\sum_{i = 1}^{i_1-2} (-1)^{i-1} \Bigg[ (\rho^i, I_1, J) \otimes \sum_{\gamma_i} (-1)^{|\gamma_i|} \eta_\mathcal{I}(\widetilde{\gamma_i})(\ell) \Bigg] \\[15pt]
    & \hspace{-0.5in} + \displaystyle\sum_{k=1}^{m-1} \sum_{i = i_k+1}^{i_{k+1}-2} (-1)^{i-1} \Bigg[ (\rho^i, I_{k+1}, J) \otimes \sum_{\gamma_i} (-1)^{|\gamma_i|} \eta_\mathcal{I}(\widetilde{\gamma_i})(\ell) \Bigg]\\[15pt]
    & \hspace{-0.5in} + \displaystyle\sum_{i = i_m+1}^{q-n+m-1} (-1)^{i-1} \Bigg[ (\rho^i, I, J) \otimes \sum_{\gamma_i} (-1)^{|\gamma_i|} \eta_\mathcal{I}(\widetilde{\gamma_i})(\ell) \Bigg] \\[15pt]
    & \hspace{-0.5in} + \displaystyle\sum_{k=1}^m (-1)^{i_k-2} \displaystyle\sum_{h=0}^{\lambda_{i_k-1}} \Bigg[ (\rho^{i_k-1}, I_k, J_{k,h}) \otimes \sum_{\gamma_{i_k-1,h}} (-1)^{|\gamma_{i_k-1,h}|} \eta_\mathcal{I}(\widetilde{\gamma_{i_k-1,h}})(\ell) \Bigg]\\[15pt]
    & \hspace{-0.5in} + \displaystyle\sum_{k=1}^m (-1)^{i_k-1}\displaystyle\sum_{h=0}^{\lambda_{i_k+1}} \Bigg[ (\rho^{i_k}, I_{k+1}, J_{k,h}) \otimes \sum_{\gamma_{i_k,h}} (-1)^{|\gamma_{i_k,h}|} \eta_\mathcal{I}(\widetilde{\gamma_{i_k,h}})(\ell) \Bigg]
\end{array} \]
where $\mathcal{I} = (\iota_1, \dots,\iota_m)$ with $\iota_k := j_k+1+\sum^{i_k-1}_{i = 1} \lambda_i$ the overall position of $z_k$ in the configuration $(\lambda,I,J)$; $\gamma_{i_k-1,h}$ runs over all $((\lambda_{i_k-1},h),\lambda_{i_k}, j_k)$-shuffles; $\gamma_{i_k,h}$ runs over all $(\lambda_{i_k}, (\lambda_{i_k+1},h),j_k)$-shuffles; and $\gamma_i$ runs over all $(\lambda_i, \lambda_{i+1})$-shuffles for all $i \ne i_k-1, i_k$. The lift $\widetilde{\gamma_i}$ (defined similarly for $\widetilde{\gamma_{i_k-1,h}}$ and $\widetilde{\gamma_{i_k,h}}$) in this differential is the lift of the shuffle $\gamma_i$ to the copy $B_{\lambda_i + \lambda_{i+1}} \le B_{n+m}$ consisting of braids that are only nontrivial on the $\lambda_i + \lambda_{i+1}$ strands starting with the $\lambda_1 + \dots + \lambda_{i-1} +1^\mathrm{st}$. We prove the main structural theorem of this paper below.

\begin{thm}\label{thm:hom_compactification}
There is an isomorphism
\[H_* (\mathrm{Conf}_{n}(\mathbb{C}_m) \cup \{ \infty \}, \{\infty\} ; \mathcal{L}) \cong H_*(D(n,m)_* \otimes L).\]
\end{thm}

\begin{proof}
Our argument will generalize of the proof of Theorem 4.8 of \cite{h22}. Let $\widetilde{D(n,m)_*}$ be the cellular chain complex of the universal cover on $\mathrm{Conf}_n(\mathbb{C}_m)$ obtained by lifting the Fox-Neuwirth cells. It suffices to describe an identification
\[ \widetilde{D(n,m)_q} \cong \mathbb{Z} \{ (((\lambda_1, \dots, \lambda_{q-n+m}),I,J),b)|b\in B_n(\mathbb{C}_m) \} \]
as right $B_n(\mathbb{C}_m)$-representations which gives the desired description of the differentials.

The top dimensional cells of $\widetilde{D(n,m)_*}$ occur when $q = 2n$ and have the general form $(((1,\dots,1),I,(0,\dots,0)),b)$ for some $I = (i_1, \dots,i_m)$. Consider the codimension-1 faces of this cell obtained by combining the $i^\mathrm{th}$ and $i+1^\mathrm{st}$ columns, i.e., putting the $i^\mathrm{th}$ and $i+1^\mathrm{st}$ points on the same vertical line. There are two main outcomes of this operation: either the $i^\mathrm{th}$ point lies below the $i+1^\mathrm{st}$ point, or vice versa. Each of these are divided into subcases, depending on whether a fixed point is involved. Recall that for any configuration in $\mathrm{Conf}_n(\mathbb{C}_m)$, the lexicographic order of points in the configuration is obtained by indexing them from bottom to top for each subsequent column starting with the leftmost one. We then label the braid element of a face based on its effect on this order of points in the configuration as follows. If the lexicographic order is preserved, we apply $\eta_I(\mathrm{id}) = \mathrm{id}$ to $b$ on the left; the codimension-1 faces in this case are labelled by:
\begin{center}
\addtolength{\leftskip}{-2cm}
\addtolength{\rightskip}{-2cm}
{\renewcommand{\arraystretch}{1.5}%
\begin{tabular}{| c | c |}
\hline
 Case & $i^\mathrm{th}$ point lies below $i+1^\mathrm{st}$ point \\
 \hline
 $1 \le i < i_1-1$ & $(((1,\dots,1,2^{(i)},1,\dots,1),I_1,(0,\dots,0)),b)$ \\
 \hline
 $i_k < i < i_{k+1}-1$ $(1 \le k < m)$ & $(((1,\dots,1,2^{(i)},1,\dots,1),I_{k+1},(0,\dots,0)),b)$ \\
 \hline
 $i_m < i < n+m$ & $(((1,\dots,1,2^{(i)},1,\dots,1),I,(0,\dots,0)),b)$ \\
 \hline
 $i = i_k-1$ $(1\le k \le m)$ & $(((1,\dots,1,2^{(i_k-1)},1,\dots,1),I_k,(0,\dots,0,1^{(k)},0,\dots,0)),b)$ \\
 \hline
 $i = i_k$ $(1 \le k \le m)$ & $(((1,\dots,1,2^{(i_k)},1,\dots,1),I_{k+1},(0,\dots,0)),b)$ \\
 \hline
\end{tabular}}
\end{center}
where $(1,\dots,1,2^{(i)},1,\dots,1)$ denotes the composition of $n+m$ where the only non-1 part is $\lambda_i = 2$, and $(0,\dots,0,1^{(k)},0,\dots,0)$ denotes an $m$-tuple where the only nonzero entry is $1$ at the $k^\mathrm{th}$ coordinate. On the other hand, if the lexicographic order changes, we apply $\eta_I(\widetilde{\gamma})$ where $\gamma$ is the corresponding permutation. The labelling system in this case is given by
\begin{center}
\addtolength{\leftskip}{-2cm}
\addtolength{\rightskip}{-2cm}
{\renewcommand{\arraystretch}{1.5}%
\begin{tabular}{| c | c |}
\hline
 Case & $i^\mathrm{th}$ point lies above $i+1^\mathrm{st}$ point \\
 \hline
 $1 \le i < i_1-1$ & $(((1,\dots,1,2^{(i)},1,\dots,1),I_1,(0,\dots,0)),\eta_I(\sigma_i) b)$ \\
 \hline
 $i_k < i < i_{k+1}-1$ $(1 \le k < m)$ & $(((1,\dots,1,2^{(i)},1,\dots,1),I_{k+1},(0,\dots,0)),\eta_I (\sigma_i) b)$ \\
 \hline
 $i_m < i < n+m$ & $(((1,\dots,1,2^{(i)},1,\dots,1),I,(0,\dots,0)),\eta_I(\sigma_i)b)$ \\
 \hline
 $i = i_k-1$ $(1\le k \le m)$ & $(((1,\dots,1,2^{(i_k-1)},1,\dots,1),I_k,(0,\dots,0)),\eta_I(\sigma_{i_k-1})b)$ \\
 \hline
 $i = i_k$ $(1 \le k \le m)$ & $(((1,\dots,1,2^{(i_k)},1,\dots,1),I_{k+1},(0,\dots,0,1^{(k)},0,\dots,0)),\eta_I(\sigma_{i_k})b)$ \\
 \hline
\end{tabular}}
\end{center}
Note that this choice of labelling is consistent with the right action of $B_n(\mathbb{C}_m)$.

More generally, a generic cell $((\lambda,I,J),b)$ corresponds to the face of the top dimensional cell $(((1,\dots,1),\mathcal{I},(0,\dots,0)),b)$ obtained by putting points into columns according to the configuration $\lambda$ while preserving the lexicographic order of points; here $\mathcal{I} = (\iota_1, \dots, \iota_m)$ records the overall positions of the fixed points in the configuration, i.e., $\iota_k = j_k+1+\sum_{i=1}^{i_k-1} \lambda_i$ for $1 \le k \le m$. However, if we arrange the face so that the lexicographic order is altered by a permutation $\gamma$, we need to multiply the element of $B_n(\mathbb{C}_m)$ in the cell's label on the left with $\eta_\mathcal{I}(\widetilde{\gamma})$. Note that this labelling system is compatible with the decomposition of braid elements into generators precisely because $\eta_\mathcal{I} (ba) = \eta_{\underline{a}(\mathcal{I})}(b) \eta_\mathcal{I} (a)$ for any $a,b \in B_{n+m}$.

It follows from this labelling system that the face maps of the complex $\widetilde{D(n,m)_*}$ are given by
\[\begin{array}{l}
d_i ((\lambda,I,J),b) = \\[5pt]
\hspace{.3in} \begin{cases}
\displaystyle \sum_{\gamma_i}(-1)^{|\gamma_i|}((\rho^i,I_1,J),\eta_\mathcal{I}(\widetilde{\gamma_i}) b) \hspace{3.5cm} \hfill 1 \le i < i_1-1\\[15pt]
\displaystyle \sum_{\gamma_i}(-1)^{|\gamma_i|}((\rho^i,I_{k+1},J),\eta_\mathcal{I}(\widetilde{\gamma_i}) b) \hfill i_k < i < i_{k+1}-1\\[15pt]
\displaystyle \sum_{\gamma_i}(-1)^{|\gamma_i|}((\rho^i,I,J),\eta_\mathcal{I}(\widetilde{\gamma_i}) b) \hfill i_m < i \le q-n+m-1\\[15pt]
\displaystyle \sum_{h=0}^{\lambda_{i_k-1}}\sum_{\gamma_{i_k-1,h}}(-1)^{|\gamma_{i_k-1,h}|}((\rho^{i_k-1},I_k,J_{k,h}),\eta_\mathcal{I}(\widetilde{\gamma_{i_k-1,h}}) b) \hfill i = i_k-1\\[15pt]
\displaystyle \sum_{h=0}^{\lambda_{i_k+1}}\sum_{\gamma_{i_k,h}}(-1)^{|\gamma_{i_k,h}|}((\rho^{i_k},I_{k+1},J_{k,h}),\eta_\mathcal{I}(\widetilde{\gamma_{i_k,h}}) b) \hfill i = i_k.
\end{cases}
\end{array}\]
The signs in the formula come from the orientations of the cells. The differential $d : \widetilde{D(n,m)_q} \to \widetilde{D(n,m)_{q-1}}$ is given by the alternating sum of the face maps: $d = \sum_{i=1}^{q-n+m-1} (-1)^{i-1} d_i$.

Given a $B_n(\mathbb{C}_m)$-representation $L$, we want to give a description of the chain complex $\widetilde{D(n,m)_*} \otimes_{\mathbb{Z}B_n(\mathbb{C}_m)} L$ and its differential. Observe that we may identify $((\lambda,I,J),b) \otimes \ell$ with $((\lambda,I,J),1) \otimes b(\ell)$, hence there is a natural isomorphism of $\mathbf{k}$-modules $\widetilde{D(n,m)_*} \otimes_{\mathbb{Z}B_n(\mathbb{C}_m)} L \cong D(n,m)_* \otimes L$. Furthermore, this identification when applied to the face maps and the differential results in the desired formula for the differential of $D(n,m)_* \otimes L$; for instance, when $1 \le i < i_1-1$,
\[\begin{split}
d_i(((\lambda,I,J),b) \otimes \ell) & = \displaystyle \sum_{\gamma_i}(-1)^{|\gamma_i|}((\rho_i,I_1,J),\eta_\mathcal{I}(\widetilde{\gamma_i}) b) \otimes \ell\\
& = \displaystyle \sum_{\gamma_i}(-1)^{|\gamma_i|}((\rho_i,I_1,J),1) \otimes \eta_\mathcal{I}(\widetilde{\gamma_i})b(\ell).
\end{split}\]
This concludes our proof of the theorem.
\end{proof}

This result provides a general framework for computing the cellular homology of $\mathrm{Conf}_n(\mathbb{C}_m)$ with local coefficients. An easy observation is that the cellular homology of $\mathrm{Conf}_n(\mathbb{C}_m) \cup \{\infty\}$ concentrates only in certain degrees.

\begin{cor}
For an arbitrary local system $\mathcal{L}$ over $\mathrm{Conf}_n(\mathbb{C}_m)$,
\[H_j (\mathrm{Conf}_{n}(\mathbb{C}_m) \cup \{ \infty \}, \{\infty\} ; \mathcal{L}) = 0\]
for all $j < n$ or $j > 2n$.
\end{cor}

\begin{proof}
Given a cell $e_{(\lambda,I,J)}$ of dimension $n+l(\lambda)-m$, the length of the partition $\lambda$ varies from $m$ to $n+m$, and hence $n\le \dim(e_{(\lambda,I,J)}) \le 2n$. It follows that there are no cells of dimension $j < n$ or $j > 2n$, so the homology vanishes at these degrees.
\end{proof}


\subsection{Algebraic analog of column configurations}

In this subsection, we will develop and study a chain complex that will capture the structure of the vertical columns in our cellular stratification of $\mathrm{Conf}_n(\mathbb{C}_m)$, in particular when specialized to the cellular chain complex of $\mathrm{Conf}_n(\mathbb{C}_m)$ with coefficients in the local system associated with the representation of $B_n(\mathbb{C}_m)$ on $V^{\otimes n} \otimes W^{\otimes m}$ discussed in Section~\ref{sec:rep_br}.

\begin{defn}
Given an associative $\mathbf{k}$-algebra $A$ and an ordered set $\mathcal{M} = (M_1, \dots, M_m)$ where each $M_k$ is an $A$-bimodule, the \textit{chain complex} $F_*(\mathcal{M},A)$ is defined at degree $q \ge m$ by
\[ F_q(\mathcal{M},A) = \displaystyle \bigoplus_{\mathcal{I}} A^{\otimes i_1-1}\otimes M_1 \otimes A^{\otimes i_2-i_1-1} \otimes M_2 \otimes \dots \otimes M_m \otimes A^{\otimes q-i_m}\]
where $\mathcal{I} = (i_1,\dots,i_m)$ for $1 \le i_1 < \dots < i_m \le q$. The face maps for $1\le i \le q-1$ are given by
\[\begin{array}{l}
d_i (a_1 \otimes \dots \otimes a_{i_k-1} \otimes \mu_k \otimes a_{i_k+1} \otimes \dots \otimes a_q) =\\
\hspace{.3in} \begin{cases}
a_1 \otimes \dots \otimes a_i a_{i+1} \otimes \dots \otimes a_q \hspace{1in} i \ne i_k-1, i_k\\
a_1 \otimes \dots \otimes a_{i_k-1} \mu_{k} \otimes \dots \otimes a_q \hfill m = i_k-1 \not\in \mathcal{I}\\
a_1 \otimes \dots \otimes \mu_k a_{i_k+1} \otimes \dots \otimes a_q \hfill m = i_k \text{ and } i_k+1 \not\in \mathcal{I}\\
0 \hfill \{i,i+1\} \subseteq \mathcal{I}
\end{cases} \end{array}\]
and differential is $d = \sum_{i=1}^{q-1} (-1)^{i-1} d_i$.
\end{defn}

The graded group structure and the differential of $F_*(\mathcal{M},A)$ are similar to that of the extended two-sided bar complex $B^e_*(A,A,A)$ of the algebra $A$ (see, e.g., \cite{gin05,etw17}). Therefore, $F_*(\mathcal{M},A)$ forms a well-defined chain complex in a similar manner as the bar complex $B^e_*(A,A,A)$. Observe that the last case of the face maps guarantees that there will always be $m$ $A$-bimodules present in every summand of $F_*(\mathcal{M},A)$.

Let $I$ be the \textit{augmentation ideal} of $A$, consisting of elements of positive degree, and let $F_*(\mathcal{M},I)$ denote the chain complex obtained by replacing all copies of $A$ in $F_*(\mathcal{M},A)$ with $I$ (one may think of $F_*(\mathcal{M},I)$ as a reduced form of the chain complex $F_*(\mathcal{M},A)$). If the bimodules $M_k$ are graded, for an element $f = a_1 \otimes \dots \otimes \mu_k \otimes \dots \otimes a_q$ with $a_i$ homogeneous elements of $A$ of degree deg$(a_i)$, we may define the degree of $f$ to be $\mathrm{deg}(f) := \sum_k \mathrm{deg}(\mu_k) + \sum_{i\ne i_k} \mathrm{deg}(a_i)$. The differential in $F_*(\mathcal{M},I)$ strictly preserves the degree of elements, hence we may define the split subcomplex generated by homogeneous elements of $F_*(\mathcal{M},I)$ of degree precisely $n$, denoted by $F_{*,n}(\mathcal{M},I)$.

Given an associative $\mathbf{k}$-algebra $A$, let $A^{op}$ denote the opposite algebra, and $A^e := A \otimes A^{op}$ be the enveloping algebra of $A$. There is a canonical isomorphism $(A^e)^{op} \cong A^e$, thus an $A$-bimodule can be regarded as a left (or equivalently, right) $A^e$-module \cite{gin05}. For the case $m = 1$, in \cite{h22} we computed the homology of the chain complex $F_*(M,I) := F_*(\mathcal{M},I)$ where $\mathcal{M}$ contains only one $A$-bimodule $M$:

\begin{prop}[\cite{h22}]
$H_*(F_{*}(M,I)) \cong \mathrm{Tor}^{A^e}_{*-1}(M,\mathbf{k})$.
\end{prop}

A key observation made in the proof of this proposition was that there is an inclusion of the chain complex $F_*(M,I)$ into the Hochschild chain complex $CH_*(A,M)$, defined degree-wise by $CH_q(A,M) = M \otimes A^{\otimes q}$, with face maps
\[d_i (\mu \otimes a_1 \otimes \dots \otimes a_q) = \begin{cases}
\mu a_1 \otimes a_2 \otimes \dots \otimes a_q \quad \hspace{1.6in}\hfill i = 0\\
\mu \otimes a_1 \otimes \dots \otimes a_i a_{i+1} \otimes \dots \otimes a_q \hfill 1 \le i \le q-1 \\
a_q \mu \otimes a_1 \otimes \dots \otimes a_{q-1} \hfill i = q\\
\end{cases}\]
and differential $d = \sum_{i=0}^{q} (-1)^i d_i$. The map $f_*: F_*(M,I) \to CH_*(A,M)$ that sends
\[a_1 \otimes \dots \otimes a_{i-1} \otimes \mu_i \otimes a_{i+1} \otimes \dots \otimes a_q \mapsto (-1)^{q(i-1)} \mu_i \otimes a_{i+1} \otimes \dots \otimes a_q \otimes 1 \otimes a_1 \otimes \dots \otimes a_{i-1}\]
is an injective map whose image is a subcomplex of $CH_*(A,M)$ which in degree $q$ has the form $\bigoplus^q_{i=1} M \otimes I^{\otimes q-i} \otimes \mathbf{k} \otimes I^{\otimes i-1}$. Denote this chain complex by $Z_*(M,I)$.

We will develop analogues of these chain complexes that involve multiple bimodules (i.e., $m\ge 2$).
\begin{defn}
Given an ordered set $\mathcal{M} = (M_1, \dots, M_m)$ of $A$-bimodules, \textit{the multi-module Hochschild chain complex} $CH_*(A,\mathcal{M})$ is defined degree-wise by
\[CH_q(A,\mathcal{M}) = \bigoplus_\mathcal{I} M_1 \otimes A^{\otimes i_2-2} \otimes M_2 \otimes A^{\otimes i_3-i_2-1} \otimes \dots \otimes M_m \otimes A^{\otimes q-i_m}\]
for $\mathcal{I} = (i_1,i_2,\dots,i_m)$ where $1 = i_1 < i_2 <\dots < i_m \le q$, with face maps
\[\begin{array}{l}
d_i (\mu_1 \otimes a_2 \otimes \dots \otimes a_{i_k-1} \otimes \mu_{i_k} \otimes a_{i_k+1} \otimes \dots \otimes a_q) =\\
\hspace{.3in} \begin{cases}
\mu_1 a_2 \otimes a_3 \otimes \dots \otimes a_{i_k-1} \otimes \mu_{i_k} \otimes a_{i_k+1} \otimes \dots \otimes a_q \quad \hspace{1in}\hfill i = 1\\
\mu_1 \otimes a_2 \otimes \dots \otimes a_i a_{i+1} \otimes \dots \otimes a_q \hfill i \ne i_k-1, i_k \\
\mu_1 \otimes a_2 \otimes \dots \otimes a_{i_k-1} \mu_{k} \otimes \dots \otimes a_q \hfill i = i_k-1 \not\in \mathcal{I}\\
\mu_1 \otimes a_2 \otimes \dots \otimes \mu_k a_{i_k+1} \otimes \dots \otimes a_q \hfill i = i_k \text{ and } i_k+1 \not\in \mathcal{I}\\
a_q \mu_1 \otimes a_2 \otimes \dots \otimes a_{i_k-1} \otimes \mu_{i_k} \otimes a_{i_k+1} \otimes \dots \otimes a_{q-1} \hfill i = q \not\in \mathcal{I}\\
0 \hfill \{i,i+1\} \subseteq \mathcal{I}
\end{cases} \end{array}\]
and differential $d = \sum_{i=1}^{q} (-1)^{i-1} d_i$.
\end{defn}

Similar to the case $m=1$, we may define an chain map $f_*: F_*(\mathcal{M},I) \to CH_*(A,\mathcal{M})$ that sends $a_1 \otimes \dots \otimes a_{i_k-1} \otimes \mu_k \otimes a_{i_k+1} \otimes \dots \otimes a_q$ to
\[(-1)^{q(i_1-1)} \mu_1 \otimes a_{i_1+1} \otimes \dots \otimes a_{i_k-1} \otimes \mu_k \otimes a_{i_k+1} \otimes \dots \otimes a_q \otimes 1 \otimes a_1 \otimes \dots \otimes a_{i_1-1}.\]
Note that for all $i \ne i_k$, $\mathrm{deg}(a_i) > 0$ since each $a_i$ is in the augmentation ideal $I$ of $A$. It follows that there is precisely one occurrence of the unit in the image of an arbitrary element of $F_*(\mathcal{M},I)$ under this map, and its position is determined by the indices $q$ and $i_1$. We deduce that this map is injective. The isomorphic image of $F_*(\mathcal{M},I)$ via this map forms a subcomplex $Z_*(\mathcal{M},I)$ of the multi-module Hochschild chain complex which in degree $q$ has the form
\[Z_q(\mathcal{M},I) = \bigoplus_{\mathcal{I}} M_1 \otimes I^{\otimes i_2-i_1-1} \otimes M_2 \otimes \dots \otimes M_m \otimes I^{\otimes q-i_m} \otimes \mathbf{k} \otimes I^{\otimes i_1-1}.\]
It suffices to compute the homology of this chain complex, which will be our focus for the rest of this subsection.

\textit{The case $m = 2$.} The key observation is that the chain complex $Z_*(\mathcal{M},I)$ is isomorphic to the totalization of an $m$-simplicial object. We give the treatment for the case $m = 2$ as an example.

\begin{figure}[t]
\begin{centering}
\[ \begin{tikzcd}
    {} 
        & \dots \arrow{d}{d^v} 
        & \dots \arrow{d}{d^v} 
    & {} \\
    \dots \arrow{r}{d^h}
        & \begin{tabular}{c}
            $(M_1 \otimes I \otimes M_2) \otimes I^{\otimes 2} \otimes \mathbf{k}$ \\[2pt]
            $\oplus (M_1 \otimes I \otimes M_2) \otimes I \otimes \mathbf{k} \otimes I$ \\[2pt]
            $\oplus (M_1 \otimes I \otimes M_2) \otimes \mathbf{k} \otimes I^{\otimes 2}$
        \end{tabular} \arrow{r}{d^h}\arrow{d}{d^v}    
        & \begin{tabular}{c}
        
            $(M_1 \otimes M_2) \otimes I^{\otimes 2} \otimes \mathbf{k}$ \\[2pt]
            $\oplus (M_1 \otimes M_2) \otimes I \otimes \mathbf{k} \otimes I$ \\[2pt]
            $\oplus (M_1 \otimes M_2) \otimes \mathbf{k} \otimes I^{\otimes 2}$
        \end{tabular} \arrow{r}\arrow{d}{d^v} 
    & 0 \\
    \dots \arrow{r}{d^h}
        & \begin{tabular}{c}
            $(M_1 \otimes I \otimes M_2) \otimes I \otimes \mathbf{k}$ \\[2pt]
            $\oplus (M_1 \otimes I \otimes M_2) \otimes \mathbf{k} \otimes I$
        \end{tabular} \arrow{r}{d^h}\arrow{d}{d^v}
        & \begin{tabular}{c}
            $(M_1 \otimes M_2) \otimes I \otimes \mathbf{k}$ \\[2pt]
            $\oplus (M_1 \otimes M_2) \otimes \mathbf{k} \otimes I$
        \end{tabular}
        \arrow{r}\arrow{d}{d^v}
    & 0 \\
    \dots \arrow{r}{d^h}
        & (M_1 \otimes I \otimes M_2) \otimes \mathbf{k} \arrow{r}{d^h} \arrow{d}
        & (M_1 \otimes M_2) \otimes \mathbf{k} \arrow{r} \arrow{d}
    & 0\\
    {}
        & 0
        & 0
    & {}
\end{tikzcd}\]
    \captionsetup{justification=centering}
    \caption{The double complex $C_{p,q}$.}
\end{centering}
\end{figure}

Construct a double complex $C_{*,*}$ by setting $C_{p,q} := Z_q(B_{p-1}(M_1,I,M_2),I) = \bigoplus_{i=1}^q (M_1 \otimes I^{\otimes p-1}\otimes M_2) \otimes I^{\otimes q-i} \otimes \mathbf{k} \otimes I^{\otimes i-1}$ (see Figure 5). The vertical face maps are given by
\[\begin{array}{l}
d^v_j ((\mu_0 \otimes a_1 \otimes \dots \otimes a_{p-1} \otimes \mu_p) \otimes b_1 \otimes \dots \otimes b_q) =\\
\hspace{.3in} \begin{cases}
(\mu_0 \otimes a_1 \otimes \dots \otimes a_{p-1} \otimes \mu_p b_1) \otimes b_2 \otimes \dots \otimes b_q \quad \hspace{1.4in}\hfill j = 0\\
(\mu_0 \otimes a_1 \otimes \dots \otimes a_{p-1} \otimes \mu_p) \otimes b_1 \otimes \dots \otimes b_j b_{j+1} \otimes \dots \otimes b_q \hfill 1 \le j \le q-1 \\
(b_q \mu_0 \otimes a_1 \otimes \dots \otimes a_{p-1} \otimes \mu_p) \otimes b_1 \otimes \dots \otimes b_{q-1} \hfill j = q
\end{cases} \end{array}\]
and the vertical differential is $d^v = \sum_{j=0}^q d^v_j$. Similarly, the horizontal face maps are given by
\[\begin{array}{l}
d^h_j ((\mu_0 \otimes a_1 \otimes \dots \otimes a_{p-1} \otimes \mu_p) \otimes b_1 \otimes \dots \otimes b_q) =\\
\hspace{.3in} \begin{cases}
(\mu_0 a_1 \otimes a_2 \otimes \dots \otimes a_{p-1} \otimes \mu_p) \otimes b_1 \otimes \dots \otimes b_q \quad \hspace{1.4in}\hfill j = 0\\
(\mu_0 \otimes a_1 \otimes \dots \otimes a_j a_{j+1} \otimes \dots \otimes a_{p-1} \otimes \mu_p) \otimes b_1 \otimes \dots \otimes b_q \hfill 1 \le j \le p-2 \\
(\mu_0 \otimes a_1 \otimes \dots \otimes a_{p-2} \otimes a_{p-1} \mu_p) \otimes b_1 \otimes \dots \otimes b_q \hfill j = p-1
\end{cases} \end{array}\]
and the horizontal differential is $d^h = \sum_{j=0}^{p-1} d^h_j$.

Observe that the associated total complex $\mathrm{Tot}(C)_*$ of $C_{*,*}$ at degree $n$ has the form \[\begin{array}{r l}
     \mathrm{Tot}(C)_n & = \displaystyle\bigoplus_{p+q=n} C_{p,q} = \bigoplus_{p+q=n} (M_1 \otimes I^{\otimes p-1}\otimes M_2) \otimes I^{\otimes q-i} \otimes \mathbf{k} \otimes I^{\otimes i-1}\\[15pt]
     & = Z_n((M_1,M_2),I).
\end{array}\]
For the differential at degree $n$, we have 
\[\begin{array}{r l}
     d & = d^h + (-1)^p d^v = d_0^h - d_1^h + \dots + (-1)^{p-1} d^h_{p-1} + (-1)^p (d^v_0 - d^v_1 + \dots + (-1)^q d^v_q)\\[5pt]
     & = d^Z_1 - d^Z_2 + \dots + (-1)^{p-1}d^Z_p + (-1)^p d^Z_{p+1} + \dots + (-1)^{n}d^Z_{n+1} = d^Z
\end{array}\]
thus indeed $Z_*((M_1,M_2),I)$ is isomorphic to the total complex of the double complex $C_{*,*}$. Notice that the horizontal filtering of $C_{*,*}$ is by the reduced bar complex $B_{*-1}(M_1, I, M_2)$, whereas the vertical filtering is by the complex $Z_*(M,I)$. Recall that the complex $Z_*(M,I)$ is isomorphic to $F_*(M,I)$, whose homology is given by $H_*(Z_*(M,I)) \cong H_*(F_*(M,I)) \cong \left( M \underset{A^e}{\overset{L}{\otimes}} \mathbf{k} \right)[-1]$ (shifted in homological degree by $-1$). It follows that

\begin{prop}
$H_*(Z_*((M_1,M_2),I)) \cong \left( \left(M_1 \underset{A}{\overset{L}{\otimes}} M_2 \right) \underset{A^e}{\overset{L}{\otimes}} \mathbf{k} \right) [-2]$.
\end{prop}

\textit{General case.} We can generalize the construction above for any $m \ge 3$. Let $\mathcal{M} = (M_1, \dots, M_m)$. Instead of a bisimplicial object $C_{*,*}$, we construct an $m$-simplicial object $C_{*,\dots,*}$ which for each $m$-tuple $(p_1,\dots,p_m)$ is given by 
\[ \begin{array}{l}
     C_{p_1,\dots,p_m} := Z_{p_m}(B_{p_{m-1}-1}((\dots B_{p_2-1}(B_{p_1-1}(M_1,I,M_2),I,M_3)\dots),I,M_m),I) \\[5pt]
     = \displaystyle \bigoplus_{i=1}^{p_m} \null ((\dots((M_1 \otimes I^{\otimes p_1-1}\otimes M_2) \otimes I^{\otimes p_2-1} \otimes M_3) \otimes \dots) \otimes I^{\otimes p_m-1} \otimes M_m) \\[5pt]
     \hspace{0.5in} \null \otimes I^{\otimes p_m-i} \otimes \mathbf{k} \otimes I^{\otimes i-1}.
\end{array} \]
The face maps in the $i^\mathrm{th}$ direction are defined similarly to the horizontal mappings in the previous case if $1 \le i \le m-1$ and to the vertical mappings if $i = m$. From this specification, we have an analogous observation about the filtering of the $m$-complex constructed from $C_{*,\dots,*}$: the filtering in the $i^{\mathrm{th}}$ direction for $1 \le i \le m-1$ is by a reduced bar complex (with degree shift $-1$), whereas that in the $m^\mathrm{th}$ direction is by the complex $Z_*(M,I)$ for some $M$. The totalization of $C_{*,\dots,*}$ is given degree-wise by $\mathrm{Tot}(C)_n = \bigoplus_{\sum p_i=n} C_{p_1,\dots,p_m} = Z_n(\mathcal{M}, I)$, which implies

\begin{thm}\label{thm:hom_F}
For $\mathcal{M} = (M_1, \dots, M_m)$,
\[H_*(F_*(\mathcal{M},I)) \cong \left\{ \left( \left( \dots\left( \left(M_1 \underset{A}{\overset{L}{\otimes}} M_2 \right) \underset{A}{\overset{L}{\otimes}} M_3 \right) \dots \right) \underset{A}{\overset{L}{\otimes}} M_m \right) \underset{A^e}{\overset{L}{\otimes}} \mathbf{k} \right\} [-m]. \]
\end{thm}


\subsection{Homology with mixed braid coefficients}

Let $(V,W,\tau,\varphi)$ be a left-separable mixed-braided vector space, and let $\mathfrak{A} := \mathfrak{A}(V)$ be the quantum shuffle algebra associated with the braided vector space $V$. Suppose $(V,W,\tau,\varphi)$ satisfies Equation~\ref{prop:comm_diag_cond}.2: $(\mathrm{id} \otimes \tau) \circ (\varphi \otimes \mathrm{id}) = (\varphi \otimes \mathrm{id}) \circ (\mathrm{id} \otimes \sigma_W) \circ (\tau \otimes \mathrm{id}) \circ (\mathrm{id} \otimes \sigma_W^{-1})$. In \cite{h22}, we defined an $\mathfrak{A}$-bimodule from the separable left-braided vector space $(V,W,\sigma_V,\tau,\varphi)$ as follows:

\begin{defn}\label{defn:M}
The \textit{graded} $\mathfrak{A}(V)$\textit{-bimodule} $\mathfrak{M}$ is defined by
\[ \mathfrak{M} = \displaystyle \bigoplus_{q \ge 1} \bigoplus_{0 \le j \le q-1} V^{\otimes j} \otimes W \otimes V^{\otimes q-j-1}.\]
Multiplication on both sides is given by the quantum shuffle product in the following sense: for the left multiplication we have
\[[a_1 | \dots | a_p] \star [b_1 | \dots | b_j |w|b_{j+2}|\dots| b_q] = \sum_{\gamma} \widetilde{\gamma} [a_1 | \dots | a_p | b_1 | \dots | b_j | w | b_{j+2} | \dots | b_q]\]
and for the right multiplication
\[\displaystyle [a_1 | \dots |a_j|w|a_{j+2}|\dots | a_p] \star [b_1 | \dots | b_q] = \sum_{\gamma}  \widetilde{\gamma} [a_1 | \dots|a_j|w|a_{j+2}|\dots | a_p | b_1 | \dots | b_q]\]
where $\gamma$ draws from all $(p,q)$-shuffles, and $\widetilde{\gamma}$ is the lift of $\gamma$ to $B_{p+q}$. The action of the braid $\widetilde{\gamma}$ on an element $\ell \in V^{\otimes j} \otimes W \otimes V^{\otimes q-j-1}$ is given by
\[ \sigma_m(\ell) = \begin{cases}
(\mathrm{id}^{\otimes m-1} \otimes \sigma_V \otimes \mathrm{id}^{\otimes n-m})(\ell) \hspace{.3in} \hfill m \ne i-1, i \\
(\mathrm{id}^{\otimes i-2} \otimes \varphi \otimes \mathrm{id}^{\otimes n-i+1})(\ell) \hfill m = i-1 \\
(\mathrm{id}^{\otimes i-1} \otimes \tau\varphi^{-1} \otimes \mathrm{id}^{\otimes n-i})(\ell) \hfill m = i.
\end{cases}
\]
\end{defn}

Write $\epsilon$ for the braided $\mathbf{k}$-module $\epsilon = \mathbf{k}$ with braiding on $\epsilon^{\otimes 2} \cong \mathbf{k}$ given by multiplication by $-1$. For a general braided $\mathbf{k}$-module $(V,\sigma)$, write $V_\epsilon = V \otimes \epsilon$ with braiding twisted by the sign on $\epsilon$, i.e., $\sigma_{V_\epsilon} = -\sigma_V$. Observe that if $(V,W,\tau,\varphi)$ is a left-separable mixed-braided vector space satisfying Equation~\ref{prop:comm_diag_cond}.2, then $(V_\epsilon, W_\epsilon, \tau, \varphi_\epsilon)$ also forms a left-separable mixed-braided vector space with $\varphi_\epsilon := -\varphi$, which satisfies an equivalent equation:
\[(\mathrm{id} \otimes \tau) \circ (\varphi_{\epsilon} \otimes \mathrm{id}) = (\varphi_\epsilon \otimes \mathrm{id}) \circ (\mathrm{id} \otimes \sigma_{W_\epsilon}) \circ (\tau \otimes \mathrm{id}) \circ (\mathrm{id} \otimes \sigma_{W_\epsilon}^{-1}).\]
Let $\mathfrak{A} := \mathfrak{A}(V_\epsilon)$ be the quantum shuffle algebra associated with $V_\epsilon$, and $\mathfrak{M}$ be the $\mathfrak{A}$-bimodule defined from $(V_\epsilon, W_\epsilon,\tau, \varphi_\epsilon)$ as described above. Let $\mathfrak{I}$ be the augmentation ideal of $\mathfrak{A}$, consisting of elements of positive degree. The following proposition shows the relationship between the algebraic structures we developed in the previous subsection and the cellular chain complex of configuration spaces with twisted coefficients.

\begin{prop}\label{prop:F=D}
Let $\mathcal{M} = (\mathfrak{M}_1, \dots, \mathfrak{M}_m)$ where each $\mathfrak{M}_k = \mathfrak{M}$. Then there is an isomorphism of chain complexes
\[F_{*,n+m}(\mathcal{M},\mathfrak{I}) \cong D(n,m)_{*+n-m} \otimes (V^{\otimes n}\otimes W^{\otimes m}).\]
\end{prop}

\begin{proof}
Observe that $F_{q,n+m}(\mathfrak{M},\mathfrak{I})$ consists of all spaces of the form
\[ V^{\otimes \lambda_1} \otimes \dots \otimes V^{\otimes \lambda_{i_k-1}} \otimes (V^{\otimes j_k} \otimes W_k \otimes V^{\otimes \lambda_{i_k}-j_k- 1}) \otimes V^{\otimes \lambda_{i_k+1}} \otimes \dots \otimes V^{\lambda_{q}} \]
where $\sum \lambda_i = n+m$. This is an ordered partition $\lambda$ of $n+m$ with $q$ parts labelled by an element of $V^{\iota_1-1} \otimes W_1 \otimes V^{\otimes \iota_2 - \iota_1 -1} \otimes W_2 \otimes \dots \otimes W_m \otimes V^{\otimes n+m-\iota_m}$, where $\iota_k = j_k+1+\sum_{i=1}^{i_k-1} \lambda_i$ is the position of the factor $W_k$ in the tensor product; furthermore, there is at most one number $\iota_k$ in every part of the partition $\lambda$. These data about the partition $\lambda$ are in one-to-one correspondence with an element of $D(n,m)_{q+n-m}$, whereas the labelling element is identified via the isomorphism $\varphi_{\mathcal{I},n+m} : V^{\otimes n} \otimes W^{\otimes m} \to V^{\iota_1-1} \otimes W_1 \otimes V^{\otimes \iota_2 - \iota_1 -1} \otimes W_2 \otimes \dots \otimes W_m \otimes V^{\otimes n+m-\iota_m}$ where $\mathcal{I} = (\iota_1, \dots, \iota_m)$. Hence there is an isomorphism of $\mathbf{k}$-modules between $F_{q,n+m}(\mathcal{M},\mathfrak{I})$ and $D(n,m)_{q+n-m} \otimes (V^{\otimes n}\otimes W^{\otimes m})$.

There are two main pieces of data in the boundary of an element $\lambda \otimes \ell$ in the chain complex $D(n,m)_* \otimes (V^{\otimes n}\otimes W^{\otimes m})$: the coarsening $\rho^i$ of $\lambda$ and the signed sum over all $(\lambda_i,\lambda_{i+1})$-shuffles of the actions of their lifts on $\ell$. Both are encapsulated in the differential of $F_{*,n+m}(\mathcal{M},\mathfrak{I})$: the coarsening $\rho^i$ is encoded in the choice of two multiplied elements, and the sum of the braid actions is contained in the quantum shuffle product of $\mathfrak{A}$ or the multiplication of $\mathfrak{M}_k$ by $\mathfrak{A}$. Observe that in the differential of $F_{*,n+m} (\mathcal{M},\mathfrak{I})$, the braids act only on a tensor subfactor $V^{\otimes j_k} \otimes W_k \otimes V^{\otimes \lambda_{i_k}-j_k-1}$ of $V^{\iota_1-1} \otimes W_1 \otimes V^{\otimes \iota_2 - \iota_1 -1} \otimes W_2 \otimes \dots \otimes W_m \otimes V^{\otimes n+m-\iota_m}$, whereas the corresponding braids act on the isomorphic image $V^{\otimes n} \otimes W^{\otimes m}$ of this full factor in the differential of $D(n,m)_{*+n-m} \otimes (V^{\otimes n}\otimes W^{\otimes m})$. These braid actions match precisely due to the commutativity of Diagram~\ref{eq:comm_diag}, which is equivalent to left-separability of the mixed-braided vector space $(V_\epsilon,W_\epsilon,\tau,\varphi_\epsilon)$ together with Equation~\ref{prop:comm_diag_cond}.2 as shown in Proposition~\ref{prop:comm_diag_cond}. The signs coming from $\epsilon$ encode the boundary orientations on cells in the Fox-Neuwirth/Fuks model. Via these identifications, the differentials of $F_{*,n+m}(\mathcal{M},\mathfrak{I})$ and $D(n,m)_{*+n-m} \otimes (V^{\otimes n}\otimes W^{\otimes m})$ are precisely the same formula, which shows their isomorphism as chain complexes.
\end{proof}

Combining this with Theorem~\ref{thm:hom_compactification} and Theorem~\ref{thm:hom_F}, we get a formula for the cellular homology of the $1$-point compactification of $\mathrm{Conf}_n(\mathbb{C}_m)$:

\begin{cor}\label{cor:tor_n_times}
Let $\mathcal{L}$ be the local system over $\mathrm{Conf}_n(\mathbb{C}_m)$ associated with the $B_n(\mathbb{C}_m)$-representation on $V^{\otimes n} \otimes W^{\otimes m}$. Then there is an isomorphism
\[ H_*(\mathrm{Conf}_n(\mathbb{C}_m) \cup \{ \infty \}, \{ \infty \}; \mathcal{L}) \cong \left\{ \left( \left( \dots\left( \left(\mathfrak{M}_1 \underset{\mathfrak{A}}{\overset{L}{\otimes}} \mathfrak{M}_2 \right) \underset{\mathfrak{A}}{\overset{L}{\otimes}} \mathfrak{M}_3 \right) \dots \right) \underset{\mathfrak{A}}{\overset{L}{\otimes}} \mathfrak{M}_m \right) \underset{\mathfrak{A}^e}{\overset{L}{\otimes}} \mathbf{k} \right\}[-n][n+m] \]
where the first bracket denotes the shift in homological degree and the second the internal degree part.
\end{cor}


\subsection{Homology of braid subgroups with one-dimensional coefficients}\label{subsec:H*_1d}

We revisit Example~\ref{exmp:V=W=k} when $\mathbf{k}$ is a field of characteristic $0$. Let the mixed-braided vector space $(V,W,\tau)$ be composed of one-dimensional $\mathbf{k}$-vector spaces $V=\mathbf{k}$ and $W = \mathbf{k}\{w\}$, with braidings $\sigma_V$, $\sigma_W$, and $\tau$ on $\mathbf{k}$ given by multiplications by $q$, $u$, and $p$ respectively, for some $p, q, u \in \mathbf{k}^\times$. Let $L$ be the $B_n(\mathbb{C}_m)$-representation on $V^{\otimes n} \otimes W^{\otimes m}$, and $\mathcal{L}$ the associated local system over $\mathrm{Conf}_n(\mathbb{C}_m)$. The braid action of $B_n(\mathbb{C}_m)$ on $V^{\otimes n} \otimes W^{\otimes m} \cong \mathbf{k}$ is given by $\sigma_i \mapsto q$ for all $1 \le i \le n-1$ and $\theta_{n,n+k} = \sigma_{n+k-1} \dots \sigma_{n+1} \tau_n \sigma_{n+1}^{-1} \dots \sigma_{n+k-1}^{-1} \mapsto u^{k-1} p (u^{-1})^{k-1} = p$ for $1 \le k \le m$. In this case, the quantum shuffle algebra $\mathfrak{A} = \mathfrak{A}(V)$ constructed from the braided vector space $(V,\sigma_V)$ is generated (as a $\mathbf{k}$-module) by the classes $x_n = [1|\dots|1]_n$, where there are $n$ occurrences of $1$. It has been shown that the algebra $\mathfrak{A}$ is isomorphic as graded rings to the quantum divided power algebra $\Gamma_q[x]$ \cite{etw17}, whose structure has been previously studied by \cite{cal06}.

\begin{defn}
    The \textit{quantum divided power algebra} $\Gamma_q[x]$ associated to $q \in \mathbf{k}^\times$ is additively generated by elements $x_n$ in degree $n$, equipped with the product
    \[ x_n \star x_m := {n+m \choose m}_q x_{n+m} \]
    where the quantum binomial coefficient is defined by
    \[ {a \choose b}_q = \dfrac{[a]_q [a-1]_q \cdots [a-b+1]_q}{[b]_q[b-1]_q \cdots [1]_q}; \quad \text{here} \quad [r]_q = \dfrac{1-q^r}{1-q} = 1+ q + \cdots + q^{r-1}. \]
\end{defn}

The isomorphism between $\Gamma_q[x]$ and $\mathfrak{A}$ sends the class $x_n$ to $[1|\dots|1]_n$ in $\mathfrak{A}$ \cite{etw17}. The following identification of the algebra $\Gamma_q[x]$ is due to Callegaro.

\begin{prop}[\cite{cal06}]\label{prop:q_div_pow_alg}
If $q$ is not a root of unity in $\mathbf{k}$, then there is an isomorphism $\Gamma_q[x] \cong \mathbf{k}[x_1]$.  If $q$ is a primitive $m^\mathrm{th}$ root of unity, then 
\[\Gamma_q[x] = \mathbf{k}[x_1]/x_1^{m} \otimes \Gamma[x_m].\]
\end{prop}

As analyzed in Example~\ref{exmp:V=W=k}, the left-braided vector space $(V,W,\sigma_V,\tau)$ is separable with the separated braiding  $\varphi := s\varphi'$ for some $s \in \mathbf{k}^{\times}$, where $\varphi'$ simply permutes two tensor factors. Recall that the $\mathfrak{A}$-bimodule $\mathfrak{M}$ is defined by
\[ \mathfrak{M} = \displaystyle \bigoplus_{n \ge 1} \bigoplus_{1 \le i \le n} V^{\otimes i-1} \otimes W \otimes V^{\otimes n-i}.\] In this case, $\mathfrak{M}$ as a $\mathbf{k}$-module is generated by the classes $y_{i,n} := [1|\dots|1|w^{(i)}|1|\dots|1]_n$ (where the superscript $i$ records the position of $w$) for all $n \ge 1$ and $1 \le i \le n$; in particular, denote $y_n := y_{n,n} = [1|\dots|1|w]_n$. Definition~\ref{defn:M} specifies the action of $B_n$ on these generators as follows:
\[ \sigma_j (y_{i,n}) = \begin{cases}
q y_{i,n} \hspace{.5in} \hfill j \ne i-1, i \\
s y_{i-1,n} \hfill j = i-1 \\
ps^{-1} y_{i+1,n} \hfill j = i.
\end{cases}
\]

In \cite{h22}, we showed that $\mathfrak{M}$ is in fact a free left (or right) $\mathfrak{A}$-module with respect to the basis $\mathcal{Y} = \{y_1, y_2, \dots\}$, and that the left and right multiplications of $\mathfrak{M}$ are related via the following formula:
\[ y_n x_m = \dfrac{1}{s^m}\sum_{h=0}^m \left[ \dfrac{1}{q^{(m-h)(n-1+h)}} {n-1+h \choose h}_q \prod_{j=0}^{h-1} \left( p - \dfrac{1}{q^{n-1+j}} \right) \right]  x_{m-h}y_{n+h}. \]
In particular, if $q$ is a primitive $m^\mathrm{th}$ root of unity, then
\[\displaystyle y_n x_m = \dfrac{1}{s^m} \left[ \left\lceil \frac{n}{m} \right\rceil \prod_{k=0}^{m-1} \left( p - \dfrac{1}{q^{n-1+k}} \right) \right]  y_{n+m} + \langle x_{h} y_{n+m-h} \rangle\]
where $\lceil x \rceil$ denotes the ceiling function of $x$, and $\langle x_{h} y_{n+m-h} \rangle$ consists of all terms of the form $x_{h} y_{n+m-h}$ for $1 \le h \le m$.

We will now proceed to compute the cellular homology of the $1$-point compactification of $\mathrm{Conf}_n(\mathbb{C}_m)$ with coefficients in the local system $\mathcal{L}$. Including the sign twist $\epsilon$ into the braidings $\sigma$ and $\varphi$ is the same as replacing $q$ and $s$ with $-q$ and $-s$, respectively. Let $\mathfrak{A}(V_\epsilon) = \Gamma_{-q}[x] =: \Gamma$. Since each $\mathfrak{M}_i = \mathfrak{M}$ is free as a $\Gamma$-module, Corollary~\ref{cor:tor_n_times} specifies to
\[ H_*(\mathrm{Conf}_n(\mathbb{C}_m) \cup \{ \infty \}, \{ \infty \}; \mathcal{L}) \cong \left\{ \left( \mathfrak{M}_1 \underset{\Gamma}{\otimes} \mathfrak{M}_2 \underset{\Gamma}{\otimes} \dots \underset{\Gamma}{\otimes} \mathfrak{M}_m \right) \underset{\Gamma^e}{\overset{L}{\otimes}} \mathbf{k} \right\}[-n][n+m]. \]
Because each $\mathfrak{M}_i$ is free as a left (or right) $\Gamma$-module with respect to a basis $\mathcal{Y}_i$, the right hand side can be simplified further to
\[ \left( \mathfrak{M}_1 \underset{\Gamma}{\otimes} \mathfrak{M}_2 \underset{\Gamma}{\otimes} \dots \underset{\Gamma}{\otimes} \mathfrak{M}_m \right) \underset{\Gamma^e}{\overset{L}{\otimes}} \mathbf{k} \cong (\mathbf{k}\mathcal{Y}_1 \otimes \mathbf{k}\mathcal{Y}_2 \otimes \dots \otimes \mathbf{k}\mathcal{Y}_m) \underset{\Gamma}{\overset{L}{\otimes}} \mathbf{k}\]
where $\mathcal{Y}_i = \{ y_1, y_2, \dots \}$ is the chosen basis for $\mathfrak{M}_i$ as a free left $\Gamma$-module.

In \cite{h22}, we fully computed this homology for all $p,q \in \mathbf{k}^\times$ when $m=1$. The computation is divided into cases depending on whether $-q$ is a root of unity and $p$ is a power of $-q$. For the purpose of the arithmetic application of this paper, we assume $-q$ is a primitive $r^\mathrm{th}$ root of unity, and demonstrate this computation when $p=-q$ and when $p$ is not a power of $-q$.

{\it Case 1: $p = -q$.}
By Proposition~\ref{prop:q_div_pow_alg}, $\Gamma = \mathbf{k}[x_1]/x_1^{r} \otimes \Gamma[x_r]$, so it suffices to study the right multiplication of the generators $y_n$ by $x_1$ and $x_r$. Let $\Lambda_{r} := \mathbf{k}[x_1]/x_1^r$ denote the degree-$r$ truncated polynomial algebra in variable $x_1$. If $\mathbf{k}$ has characteristic $0$, there is an isomorphism $\Gamma[x_r] \cong \mathbf{k}[x_r]$. Consider the multiplication by $x_r$. In this case, the above formula gives
\[\displaystyle y_n x_r = \dfrac{1}{(-s)^r}\left[ \left\lceil \frac{n}{r} \right\rceil \prod_{k=0}^{r-1} \left( p - \dfrac{1}{(-q)^{n-1+k}} \right) \right]  y_{n+r}\]
in $\mathbf{k} \mathcal{Y}$. Since the power of $-q$ in the product cycles through $r$ consecutive values, we see that $y_n x_r = 0$ if and only if $p$ is a power of $-q$. Here, we have $p = -q$, so $\mathbf{k}\mathcal{Y}$ is trivial as a $\mathbf{k}[x_r]$-module. On the other hand,
\[ y_n x_1 = \frac{1-(-q)^n}{-s(1+q)} \left( p - \frac{1}{(-q)^{n-1}} \right) y_{n+1} = \frac{1-(-q)^n}{-s(1+q)} \left( -q - \frac{1}{(-q)^{n-1}} \right) y_{n+1} \]
vanishes iff $r$ divides $n$. It follows that $\mathbf{k}\mathcal{Y}$ is freely generated by $\{ y_1, y_{r+1}, \dots\}$ as a $\Lambda_r$-module, so we have
    \[ \begin{array}{r l}
         \mathbf{k}\mathcal{Y}_1 \otimes \dots \otimes \mathbf{k}\mathcal{Y}_m & \cong \mathbf{k}\{y_1, y_2,\dots\} \otimes \dots \otimes \mathbf{k}\{y_1, y_2,\dots\} \otimes \mathbf{k}\{y_1, y_{r+1}, \dots\} [x_1]/x_1^r \\[5pt]
        & \displaystyle \cong \bigoplus_{a_1, .., a_{m-1} \ge 1, a_m \ge 0} \Lambda_r \{y_{a_1} \otimes \dots \otimes y_{a_{m-1}} \otimes y_{ra_m+1}\} \\[15pt]
        & \displaystyle \cong \bigoplus_{a_1, \dots, a_{m-1} \ge 1, a_m \ge 0} \Sigma^{a_1 + \dots + a_{m-1} + ra_m+1} \Lambda_r
    \end{array} \]
as $\Lambda_r$-modules. We then have
    \[\begin{array}{r l}
         (\mathbf{k}\mathcal{Y}_1 \otimes \mathbf{k}\mathcal{Y}_2 \otimes \dots \otimes \mathbf{k}\mathcal{Y}_m) \underset{\Gamma}{\overset{L}{\otimes}} \mathbf{k} & \cong \left( \displaystyle \bigoplus_{a_1, \dots, a_{m-1} \ge 1, a_m \ge 0} \Sigma^{a_1 + \dots + a_{m-1} + ra_m+1} \Lambda_r \right) \underset{\Lambda_r \otimes \mathbf{k}[x_r]}{\overset{L}{\otimes}} \mathbf{k}\\[15pt]
         & \displaystyle \cong \bigoplus_{a_1, \dots, a_{m-1} \ge 1, a_m \ge 0} \Sigma^{a_1 + \dots + a_{m-1} + ra_m+1} \mathbf{k} \underset{\mathbf{k}[x_r]}{\overset{L}{\otimes}} \mathbf{k} \\[15pt]
         & \displaystyle \cong \bigoplus_{a_1, \dots, a_{m-1} \ge 1, a_m \ge 0} \Sigma^{a_1 + \dots + a_{m-1} + ra_m+1} \Lambda[z_r]
    \end{array} \]
for some $z_r \in \mathrm{Tor}_{1,r}$, i.e.,
    \[ \left\{(\mathbf{k}\mathcal{Y}_1 \otimes \mathbf{k}\mathcal{Y}_2 \otimes \dots \otimes \mathbf{k}\mathcal{Y}_m) \underset{\Gamma}{\overset{L}{\otimes}} \mathbf{k} \right\}_{j,n} = \begin{cases}
        \mathbf{k}^{\oplus P(n+r-1)} \quad \text{for } j = 0\\
        \mathbf{k}^{\oplus P(n-1)} \hspace{.3in} \text{for } j = 1\\
        0 \hspace{.8in} \text{else,}
    \end{cases} \]
where $P(n)$ is the number of compositions $(a_1,\dots,a_{m-1},ra_m)$ of $n$ $(a_i \ge 1)$, which can be computed explicitly as
\[ P(n) = \sum_{a=1}^{\lfloor\frac{n}{r}\rfloor} P_{m-1}(n-ra) = \sum_{a=1}^{\lfloor\frac{n}{r}\rfloor} {n-ra-1 \choose m-2}\]
where the partition function $P_k(n) = {n-1 \choose k-1}$ counts the number of $k$-part compositions of $n$.
By applying the universal coefficient theorem and Poincar\'e duality to the dual over $\mathbf{k}$ of the left side of Corollary~\ref{cor:tor_n_times}, we have
\[\begin{split}
    H_* (\text{Conf}_{n}(\mathbb{C}_m) \cup \{ \infty \}, \{\infty\} ; \mathcal{L})^* & \cong H^* (\text{Conf}_{n}(\mathbb{C}_m) \cup \{ \infty \}, \{\infty\} ; \mathcal{L})\\
    & \cong H^*_c (\text{Conf}_{n}(\mathbb{C}_m); \mathcal{L})\\
    & \cong H_{2n-*} (\text{Conf}_{n}(\mathbb{C}_m); \mathcal{L}),
\end{split}\]
i.e., $H_*(\mathrm{Conf}_n(\mathbb{C}_m); \mathcal{L}) \cong H_{2n-*}(\mathrm{Conf}_n(\mathbb{C}_m) \cup \{ \infty \}, \{ \infty \}; \mathcal{L})^*$. It follows from the above computation that:

\begin{thm}\label{thm:H*_punct_p=-q}
\[H_j(\mathrm{Conf}_n(\mathbb{C}_m); \mathcal{L}) = \begin{cases}
        \mathbf{k}^{\oplus P(n+m+r-1)} \quad \text{for } j = n\\
        \mathbf{k}^{\oplus P(n+m-1)} \hspace{.29in} \text{for } j = n-1\\
        0 \hspace{.95in} \text{else.}
    \end{cases}\]
\end{thm}

In particular, $H_j(\mathrm{Conf}_n(\mathbb{C}_m); \mathcal{L}) = 0$ for all $j \le n-2$.

Recall that there is a fibration
\[ \mathrm{Conf}_n(\mathbb{C}_m) \to \mathrm{Conf}_{n,m}(\mathbb{C}) \to \mathrm{Conf}_m (\mathbb{C}) \]
\cite{fadneu62}, which induces a spectral sequence that computes
\[ H_p(\mathrm{Conf}_m (\mathbb{C}); H_q(\mathrm{Conf}_n(\mathbb{C}_m); \mathcal{L})) \Rightarrow H_{p+q}(\mathrm{Conf}_{n,m}(\mathbb{C}); \mathcal{F})\]
where $\mathcal{F}$ is the local system over $\mathrm{Conf}_{n,m}(\mathbb{C})$ associated with the $B_{n,m}$-representation on $V^{\otimes n} \otimes W^{\otimes m}$ (see Example~\ref{exmp:V=W=k}), i.e., $\mathcal{L} = i^*\mathcal{F}$ is the pullback of $\mathcal{F}$ by the inclusion $i:\mathrm{Conf}_n(\mathbb{C}_m)\to\mathrm{Conf}_{n,m}(\mathbb{C})$. Since $H_q(\mathrm{Conf}_n(\mathbb{C}_m);\mathcal{L}) = 0$ for all $q \le n-2$, it follows that $H_j(\mathrm{Conf}_{n,m}(\mathbb{C});\mathcal{F}) = 0$ for all $j \le n-2$. This vanishing line holds when the braiding $\sigma_W$ is given by multiplication by an arbitrary unit $u$. Now assume $u=q$ (e.g., $q = u = 1$, $p = -1$). Notice that the space $\mathrm{Conf}_{n,m}(\mathbb{C})$ is symmetric in the sense that $\mathrm{Conf}_{n,m}(\mathbb{C}) \cong \mathrm{Conf}_{m,n}(\mathbb{C})$, and the action of the representation with respect to $n$ and $m$ is also symmetric since $\sigma_V = \sigma_W$. By the same argument above, we have $H_j(\mathrm{Conf}_{n,m}(\mathbb{C});\mathcal{F}) = 0$ for all $j \le m-2$. This proves a stronger vanishing range for the cellular homology of $\mathrm{Conf}_{n,m}(\mathbb{C})$:

\begin{cor}\label{cor:H*_mix_p=-q}
$H_j(\mathrm{Conf}_{n,m}(\mathbb{C});\mathcal{F}) = 0$ for all $j \le \mathrm{max}(n,m) -2$.
\end{cor}

{\it Case 2: $p$ is not a power of $-q$}. In this case, \cite{h22} has shown that $\mathbf{k}\mathcal{Y} \cong \Sigma^1 \Gamma$, thus
\[\begin{array}{r l}
         (\mathbf{k}\mathcal{Y}_1 \otimes \mathbf{k}\mathcal{Y}_2 \otimes \dots \otimes \mathbf{k}\mathcal{Y}_m) \underset{\Gamma}{\overset{L}{\otimes}} \mathbf{k} & \cong \left( \displaystyle \bigoplus_{a_1, \dots, a_{m-1} \ge 1} \Sigma^{a_1 + \dots + a_{m-1}+1} \Gamma \right) \underset{\Gamma}{\overset{L}{\otimes}} \mathbf{k}\\[15pt]
         & \displaystyle \cong \bigoplus_{a_1, \dots, a_{m-1} \ge 1} \Sigma^{a_1 + \dots + a_{m-1}+1} \mathbf{k}.
\end{array} \]
In particular, we have
\[ \left\{(\mathbf{k}\mathcal{Y}_1 \otimes \mathbf{k}\mathcal{Y}_2 \otimes \dots \otimes \mathbf{k}\mathcal{Y}_m) \underset{\Gamma}{\overset{L}{\otimes}} \mathbf{k} \right\}_{j,n} = \begin{cases}
        \mathbf{k}^{\oplus P_{m-1}(n-1)} \quad \text{for } j = 0\\
        0 \hspace{.81in} \text{else,}
\end{cases} \]
By applying the universal coefficient theorem and Poincar\'{e} duality over $\mathbf{k}$, we have:
\begin{thm}\label{thm:H*_punct_p_not_power}
$H_j (\mathrm{Conf}_n(\mathbb{C}_m);\mathcal{L}) = \begin{cases}
        \mathbf{k}^{\oplus {n+m-2 \choose m-2}} \quad \text{ for } j=n\\
        0 \quad \text{ else.}
        \end{cases}$
\end{thm}
With this, we obtain a slightly better vanishing range for the homology of $\mathrm{Conf}_{n,m}(\mathbb{C})$ in this case, again when $u=q$.
\begin{cor}\label{cor:H*_mix_p_not_power}
$H_j(\mathrm{Conf}_{n,m}(\mathbb{C});\mathcal{F}) = 0$ for all $j \le \mathrm{max}(n,m) -1$.
\end{cor}


\section{Character sums of the resultant}\label{sec:char_sum}

In this final section, we will relate the resultant of pairs of monic square-free coprime polynomials to local systems on bicolor configuration spaces, and prove our main theorem about the character sums 
\[ F_\chi(n,m,q) = \sum_{(f,g)\in \mathrm{Conf}_{n,m}(\mathbb{F}_q)} \chi(\mathcal{R}(f,g))\]
for any nontrivial character $\chi:\mathbb{F}_q \to \mathbb{C}$ by leveraging our new topological computations.

Recall that on $\mathrm{Conf}_{n,m}$, the resultant can be interpreted as a map $\mathcal{R}: \mathrm{Conf}_{n,m} \to \mathbb{A}^1 \setminus \{0\}$. In the analytic topology over $\mathbb{C}$, the map $\mathcal{R} : \mathrm{Conf}_{n,m}(\mathbb{C}) \to \mathbb{C}^\times$ is given by the product formula
\[ \mathcal{R}(x_1, \dots, x_n, y_1, \dots, y_m) = \prod_{i,j}(x_i - y_j) \]
where $x_1,\dots,x_n, y_1, \dots, y_m$ correspond to the roots of two square-free coprime polynomials. Recall that the fundamental group $B_{n,m}$ of $\mathrm{Conf}_{n,m}(\mathbb{C})$ has three types of generators: blue-blue crossings, red-red crossings, and blue-red wraparounds. Consider the following representative of a blue-blue crossing in $\pi_1(\mathrm{Conf}_{n,m}(\mathbb{C}))$: let $x$ be a configuration with two blue points $x_i = -1, x_{i+1} = 1$, and all other points far away, and let $\alpha$ be the loop in $\mathrm{Conf}_{n,m}(\mathbb{C})$ based at $x$ given by an $e^{\pi i}$-rotation on $x_i, x_{i+1}$, and identity on the rest. Observe that each term in the product formula of the resultant has a negligible deviation as the input varies on $\alpha$, so in particular, the loop $\mathcal{R}\alpha$ is contractible. A similar observation can be made for a red-red crossing. Meanwhile, for a representative of a blue-red wraparound, let $x$ now be a configuration with a blue point $x_i = 1$, a red point $y_j = 0$, and all other points far away, and let $\alpha$ be the loop in $\mathrm{Conf}_{n,m}(\mathbb{C})$ given by an $e^{2\pi i}$-rotation on $x_i$ and identity on the rest. In this case, almost every term in the product formula of the resultant again does not change much as the input varies on $\alpha$, except for the term $x_i - y_j = e^{2\pi ti}$ for $t \in [0,1]$. This implies that the loop $\mathcal{R}\alpha$ travels around the puncture of $\mathbb{C}^\times$ precisely once. From these analyses, the homomorphism of fundamental groups $\mathcal{R}_*:B_{n,m} \to \mathbb{Z}$ induced by the resultant map sends the generators $\sigma_i$ of $B_{n,m}$ to $0$ ($1 \le i \le n+m-1, i\ne n$) and $\tau_n$ to $1$, i.e., counting the winding number of blue strands wrapping around red strands. This map agrees with the induced homomorphism on \'etale fundamental groups $\mathcal{R}_*:\widehat{B_{n,m}}\to\widehat{\mathbb{Z}}$, i.e., the diagram
\[ \begin{tikzcd}
\widehat{B_{n,m}} \arrow{r}{\mathcal{R}_*} & \widehat{\mathbb{Z}} \\%
B_{n,m} \arrow{r}{\mathcal{R}_*} \arrow[hookrightarrow]{u} & \mathbb{Z} \arrow[hookrightarrow]{u}
\end{tikzcd} \]
commutes.

Recall that the {\it Kummer sheaf} $\mathcal{L}_\chi$ associated to a character $\chi$ of $\mathbb{F}_q$ is the rank-$1$ local system on $\mathbb{A}^1_{\mathbb{F}_q}\setminus \{0\}$ defined by means of the Lang isogeny, with a special property that $\mathrm{tr}(\mathrm{Frob}_q|(\mathcal{L}_{\chi})_x) = \chi(x)$ for all $x \in \mathbb{F}_q^\times$ (see, e.g., \cite{ngo}). By pulling back this sheaf by the resultant map $\mathcal{R}:\mathrm{Conf}_{n,m}\to\mathbb{A}^1 \setminus \{0\}$, we obtain a rank-$1$ local system $\mathcal{R}^*\mathcal{L}_\chi$ on $\mathrm{Conf}_{n,m}$ such that $\mathrm{tr}(\mathrm{Frob}_q|(\mathcal{R}^*\mathcal{L}_{\chi})_{(f,g)}) = \chi(\mathcal{R}(f,g))$ for all $(f,g) \in \mathrm{Conf}_{n,m}$. It follows that the character sum $F_\chi(n,m,q)$ can be written as
\[ F_\chi(n,m,q) = \sum_{(f,g)\in \mathrm{Conf}_{n,m}(\mathbb{F}_q)} \mathrm{tr}(\mathrm{Frob}_q|(\mathcal{R}^*\mathcal{L}_{\chi})_{(f,g)}) \]
and thus, by the Grothendieck-Lefschetz trace formula with twisted coefficients, can be approached by studying the cohomology groups of $\mathrm{Conf}_{n,m}$ with coefficients in $\mathcal{R}^*\mathcal{L}_\chi$.

\begin{thm}\label{thm:main_nb}
    Let $\chi$ be a nontrivial character of $\mathbb{F}_q$. Then
    \[|F_\chi(n,m,q)| \le 2^{2n+2m-1}\frac{q^{n+m+1-\mathrm{max}(n,m)/2}-1}{\sqrt{q}-1}\]
    if $\chi$ is quadratic, and
    \[|F_\chi(n,m,q)| \le 2^{2n+2m-1}\frac{q^{n+m+(1-\mathrm{max}(n,m))/2}-1}{\sqrt{q}-1}\]
    otherwise.
\end{thm}

\begin{proof}
Observe that the local system $\mathcal{R}^*\mathcal{L}_\chi$ on $\mathrm{Conf}_{n,m}$ is associated with the representation of its \'{e}tale fundamental group $\phi:\widehat{B_{n,m}}\xrightarrow{\mathcal{R}_*}\widehat{\mathbb{Z}}\to\mathbb{Z}/(q-1)\mathbb{Z} \cong\mathbb{F}_q^\times\xrightarrow{\chi}\mathbb{C}^\times$. The corresponding local system on $\mathrm{Conf}_{n,m}(\mathbb{C})$ is therefore associated with the homomorphism $B_{n,m} \to \mathbb{C}^\times$ that sends the generators $\sigma_i$ to $1$ for $1 \le i \le n+m-1, i\ne n$, and $\tau_n$ to a primitive $d^\mathrm{th}$ root of unity $\xi$, where $d = |\chi|$. By Poincar\'{e} duality and Artin's comparison theorem, we have
\[\begin{split}
    \mathrm{dim}H_{c,\acute{e}t}^{2n+2m-i}(\mathrm{Conf}_{n,m};\mathcal{R}^*\mathcal{L}_\chi) & = \mathrm{dim}H^i_{\acute{e}t}(\mathrm{Conf}_{n,m};\mathcal{R}^*\mathcal{L}_\chi) \\
    & = \mathrm{dim}H^i(\mathrm{Conf}_{n,m}(\mathbb{C});\mathcal{R}^*\mathcal{L}_\chi).
\end{split}\]
From our topological computations in Section~\ref{subsec:H*_1d}, it follows that \[\mathrm{dim}H^i(\mathrm{Conf}_{n,m}(\mathbb{C});\mathcal{R}^*\mathcal{L}_\chi) = 0\] for all $i \le \mathrm{max}(n,m)-2$ when $d=2$ by Corollary~\ref{cor:H*_mix_p=-q}, and for $i \le \mathrm{max}(n,m)-1$ when $d>2$ by Corollary~\ref{cor:H*_mix_p_not_power}.
For other degrees, we give a bound on the dimension of the cohomology group by directly bounding the number of $(2n+2m-i)$-dimensional cells of $\mathrm{Conf}_{n,m}$; that is, 
\[ \mathrm{dim} H_{c,\acute{e}t}^{2n+2m-i}(\mathrm{Conf}_{n,m};\mathcal{R}^*\mathcal{L}_\chi) \le {n+m \choose n}{n+m-1 \choose i} \le 2^{2n+2m-1} \]
where the first binomial coefficient is the number of coloring choices of $n$ blue and $m$ red points, and the second is the number of compositions of $n+m$ of length $n+m-i$.

Let $\rho:X_\chi\to\mathrm{Conf}_{n,m}$ be the $d$-fold cover of $\mathrm{Conf}_{n,m}$ associated with $\phi : \widehat{B_{n,m}} \to \mathrm{im}(\phi) = \mu_d$, where $d=|\chi|$. Since $X_\chi$ is a finite-sheeted cover of $\mathrm{Conf}_{n,m}$, there is an injection
\[ H^*(\mathrm{Conf}_{n,m};\mathcal{R}^*\mathcal{L}_\chi) \hookrightarrow H^*(X_\chi;\mathbb{Q}_\ell).\]
The eigenvalues of the geometric Frobenius on the right hand side are bounded by Deligne's bounds, which thus also apply to the left hand side.
Letting $\mu$ be the smallest number such that $\mathrm{dim}H^i(\mathrm{Conf}_{n,m}(\mathbb{C});\mathcal{R}^*\mathcal{L}_\chi) \ne 0$, by the Grothendieck-Lefschetz trace formula with twisted coefficients we then have
\[\begin{split}
|F_\chi(n,m,q)| & = \left| \sum_{(f,g)\in \mathrm{Conf}(\mathbb{F}_q)} \mathrm{tr} (\mathrm{Frob}_q|(\mathcal{R}^*\mathcal{L}_\chi)_{(f,g)}) \right|\\
& = \left| \sum_{i=0}^{2n+2m} (-1)^i \mathrm{tr} (\mathrm{Frob}_q|H_{c,\acute{e}t}^{2n+2m-i}(\mathrm{Conf}_{n,m};\mathcal{R}^*\mathcal{L}_\chi)) \right|\\
& \le \sum_{i=0}^{2n+2m} q^{n+m-i/2}\mathrm{dim}H_{c,\acute{e}t}^{2n+2m-i}(\mathrm{Conf}_{n,m};\mathcal{R}^*\mathcal{L}_\chi)\\
& = \sum_{i=\mu}^{2n+2m} q^{n+m-i/2}\mathrm{dim}H_{c,\acute{e}t}^{2n+2m-i}(\mathrm{Conf}_{n,m};\mathcal{R}^*\mathcal{L}_\chi)\\
& \le 2^{2n+2m-1} \sum_{i=\mu}^{2n+2m} q^{n+m-i/2}\\
& = 2^{2n+2m-1}\frac{(\sqrt{q})^{2n+2m-\mu+1}-1}{\sqrt{q}-1}\\
& = 2^{2n+2m-1} \frac{q^{n+m+(1-\mu)/2}-1}{\sqrt{q}-1}.
\end{split}\]
The statement of the theorem then follows immediately from the homological vanishing range asserted by Corollaries~\ref{cor:H*_mix_p=-q} and~\ref{cor:H*_mix_p_not_power}.
\end{proof}

Notice that the factor $2^{2n+2m-1}$ resulted from our very crude bound on the dimensions of $H_{c,\acute{e}t}^*(\mathrm{Conf}_{n,m};\mathcal{R}^*\mathcal{L}_\chi)$, whose computation may significantly improve the bound for small $q$. For large $q$ however, this factor is negligible, and our bound on $|F_\chi(n,m,q)|$ for any nontrivial character $\chi$ approximates $q^{n+m+(1-\mathrm{max}(n,m))/2}$.

The number of monic square-free degree-$n$ polynomials over $\mathbb{F}_q$ is classically known to be $q^n(1-1/q)$ (see, e.g., \cite{cef_repstab}). The average of a character $\chi$ of the resultant over all pairs of monic square-free polynomials of degrees $n$ and $m$ is therefore bounded by
\[\begin{split}
    \frac{|F_\chi(n,m,q)|}{q^{n+m}(1-1/q)^2} & \le \dfrac{1}{q^{n+m}(1-1/2)^2} \cdot 2^{2n+2m-1}\dfrac{q^{n+m+1-\mathrm{max}(n,m)/2}-1}{\sqrt{q}/2}\\
    & \le 4q^{-n-m} \cdot 2^{2n+2m+1} q^{n+m+(1-\mathrm{max}(n,m))/2}\\
    & = 2^{2n+2m+3} q^{(1-\mathrm{max}(n,m))/2}
\end{split} \]
which approaches $0$ when $n$ or $m$ approaches $\infty$ if $q \gtrsim 2^8$, and at the rate of approximately $q^{(1-\mathrm{max}(n,m))/2}$ when $q$ is large. Qualitatively,

\begin{cor}\label{cor:res_avg}
    For a sufficiently large $q$, the asymptotic average of a nontrivial character of the resultant over pairs of monic square-free polynomials over $\mathbb{F}_q$ approaches $0$ as the degree of either or both polynomials grows indefinitely.
\end{cor}


\bibliographystyle{alpha}
\bibliography{BraidRef}

\end{document}